\documentclass[11pt]{article}
\usepackage{graphicx} 
\usepackage{amsmath,amssymb,bm,amsthm,mathrsfs,amscd, mathtools, url}
\usepackage{here}

\usepackage{fancyhdr}
\usepackage{tikz}
\usepackage{color}
\usepackage{algorithm}
\usepackage{algpseudocode}

\usepackage{hyperref}

\theoremstyle{plain}
\newtheorem{thm}{Theorem}[section]
\newtheorem{lem}[thm]{Lemma}
\newtheorem{cor}[thm]{Corollary}
\newtheorem{prop}[thm]{Proposition}

\theoremstyle{definition}
\newtheorem{defi}[thm]{Definition}
\newtheorem{example}[thm]{Example}

\theoremstyle{remark}
\newtheorem{rem}[thm]{Remark}

\usepackage{authblk}

\allowdisplaybreaks

\DeclareMathOperator{\diag}{diag}
\DeclareMathOperator{\GL}{GL}
\DeclareMathOperator{\lcm}{lcm}
\DeclareMathOperator{\Mat}{Mat}

\DeclareMathOperator{\rank}{rank}
\DeclareMathOperator{\supp}{supp}
\DeclareMathOperator{\wt}{wt}

\newcommand{\abs}[1]{\left\lvert #1 \right\rvert}

\numberwithin{equation}{section}

\title{Periodicity of weight enumerators for codes generated by\\ an integral matrix}
\author[1]{Koji Imamura\thanks{k-imamura@imi.kyushu-u.ac.jp}}
\author[2]{Norihiro Nakashima\thanks{nakashima@nitech.ac.jp}}
\author[3]{Takuya Saito\thanks{saito@icredd.hokudai.ac.jp}}

\affil[1]{Institute of Mathematics for Industry, Kyushu University}
\affil[2]{Department of Mathematics, Nagoya Institute of Technology}
\affil[3]{Institute for Chemical Reaction Design and Discovery, Hokkaido University}
\date{}

\begin{document}

\maketitle

\begin{abstract}
In the theory of error-correcting codes, the minimum weight and the weight enumerator play a crucial role in evaluating the error-correcting performance.
In this paper, by viewing the weight enumerator as a quasi-polynomial, we reduce the determination of the minimum weight of the resulting family to finitely many representative values of $q$.
We also give a transformation formula between the Tutte quasi-polynomial and the weight enumerator.
Furthermore, we compute the number of full-support codewords for the codes related to the special matroids $N_k$ and $Z_k$.
This is equivalent to computing the characteristic quasi-polynomials of the hyperplane arrangements related to $N_k$ and $Z_k$.

\vspace{2mm}

\end{abstract}

\section{Introduction}

Linear codes over integer residue rings $\mathbb{Z}_q=\mathbb{Z}/q\mathbb{Z}$ provide a natural generalization of linear codes over finite fields and have been studied since the work of Blake \cite{Blake1975}.
They include codes over finite chain rings $\mathbb{Z}_{p^s}$ as an important special case and arise naturally in various contexts, such as $\mathbb{Z}_4$-linear
representations of Kerdock and Preparata codes, incidence matrices of graphs, and integer representations of hyperplane arrangements.
An integral matrix gives rise, by reduction modulo $q$, to a family of linear codes over the residue rings $\mathbb{Z}_q$, making it natural to study how coding-theoretic invariants vary across the family.

The minimum weight of an error-correcting code is one of its most fundamental invariants.
It determines the error-detecting and error-correcting performance of the code and plays a central role in decoding theory and performance analysis.
However, computing the minimum weight is known to be computationally hard in general, even for binary linear codes.
Therefore, structural results that reduce the amount of computation or clarify how the minimum weight changes as $q$ varies are of considerable interest.
The main purpose of this paper is to establish such a structure by viewing the weight enumerator as a quasi-polynomial in $q$.

As Jurrius and Pellikaan summarize in \cite{JP2013}, there is a strong relationship between hyperplane arrangements and error-correcting codes.
For example, the number of codewords whose entries are all nonzero can be computed from the characteristic polynomial of the hyperplane arrangement associated with the code.
In the theory of hyperplane arrangements, the study of quasi-polynomial structures originated in the work of Kamiya, Takemura, and Terao \cite{KTT08,KTT11}, who introduced the notion of the characteristic quasi-polynomial as a quasi-polynomial refinement of the characteristic polynomial.
Since then, the subject has developed rapidly, with recent works by Higashitani, Tran, and Yoshinaga~\cite{HTY2023} focusing on periods of characteristic quasi-polynomials, including a detailed study of minimum periods and the associated lcm periods.
This work is closely related to the study of the poset of layers of the toric arrangement. In particular, following the work by Moci \cite{Moci2012}, it was shown that the last constituent of the characteristic quasi-polynomial is the characteristic polynomial of the toric arrangement by Liu, Tran, and Yoshinaga \cite{Liu-Tran-Yoshinaga21}.

Recent research on the relationship between hyperplane arrangements and error-correcting codes includes \cite{KNT2025}, which shows that the coboundary quasi-polynomial (equivalent to the weight enumerator up to multiplication by a constant) is a common refinement of the coboundary polynomial and the characteristic quasi-polynomial.
Motivated by the work of Kamiya, Takemura, and Terao, we construct a quasi-polynomial analog of the weight enumerator for error-correcting codes. We then apply this structure to reduce the determination of the minimum weight over $\mathbb{Z}_q$ to computations over finitely many representative residue rings.

We denote by $C_G(q)$ the code over $\mathbb{Z}_q$ with generator matrix $G$, where $q$ is a positive integer.
The minimum weight (or minimum Hamming weight) of a code is the smallest Hamming weight among its nonzero codewords.
We denote the minimum weight of the code $C_G(q)$ by $d_q$.
The notation appearing in Theorem \ref{thm-sufficient-necessary-mindist} will be defined in detail in Section \ref{sec-preliminaries}.
The following is the main result of this paper.
\begin{thm}\label{thm-sufficient-necessary-mindist}
    Let $G$ be a $k\times n$ matrix whose entries are integers.
    Assume that $G$ has no zero columns.
    Let $\rho_0$ be the lcm period of $G$.
    Let $e_1|e_2|\cdots|e_{r}$ be the elementary divisors of $G$, where $r$ is the rank of $G$.
    Let $m\in\mathbb{Z}_{>0}$ be a divisor of $\rho_0$.
    Then the following hold.
    \begin{itemize}
        \item[\textup{(1)}] Let $m_0 := \min\{q\in\mathbb{Z}\mid q>1,\ \gcd(q,\rho_0)=1\}$. Then for any integer $q>1$ with $\gcd(q,\rho_0)=1$, we have $d_q=d_{m_0}$.
        \item[\textup{(2)}] If $m\geq 2$ and the integers $m$ and $e_r$ are relatively prime, then $d_q=d_m$ for any $q\in\mathbb{Z}_{>0}$ with $\gcd(q,\rho_0)=m$.
        \item[\textup{(3)}] If $d_q=d_m$ for any $q\in\mathbb{Z}_{>0}$ with $\gcd(q,\rho_0)=m$, then $m\nmid e_1$, i.e., $m$ does not divide $e_1$.
    \end{itemize}
\end{thm}
When $m$ divides $q$, we have $d_m\ge d_q$ (see Remark \ref{rem-reduct-min-dist}).
Theorem \ref{thm-sufficient-necessary-mindist} is a refinement of this fact.
In the case where $m$ and $e_r$ are relatively prime for any $m \geq 2$, Theorem \ref{thm-sufficient-necessary-mindist}\textup{(1)} and \textup{(2)} imply that the values $d_q$ for all integers $q>1$ are determined by certain values among the minimum weights $d_2,d_3,\dots, d_{\rho_0+1}$.
For example, when $\rho_0=2$, calculating $d_2$ and $d_3$ allows us to determine $d_q$ for all $q\in\mathbb{Z}_{>0}$.
When $\rho_0=6$, calculating $d_2$, $d_3$, $d_5$, and $d_6$ allows us to determine $d_q$ for all $q\in\mathbb{Z}_{>0}$.
In particular, if $G$ is represented as $G=(I|A)$ with identity matrix $I$ and some matrix $A$, then all elementary divisors of $G$ equal $1$.
Hence $e_r=1$, and $e_r$ is relatively prime to every $m\ge2$.
In addition, Theorem \ref{thm-sufficient-necessary-mindist} \textup{(3)} is a necessary condition for the minimum weight to be constant.
The converses of Theorem \ref{thm-sufficient-necessary-mindist}\textup{(2)} and \textup{(3)} do not hold.
Example \ref{ex-2004} is a counterexample to the converse of Theorem \ref{thm-sufficient-necessary-mindist} \textup{(2)}, and Example \ref{ex-Kerdoc4} is a counterexample to the converse of Theorem \ref{thm-sufficient-necessary-mindist} \textup{(3)}.
See Section \ref{sec-examples} for details.
Moreover, our main results show that the minimum weights over all residue rings can be determined by finitely many minimum-weight computations.
Corollary~\ref{cor-finite-determination} provides an explicit upper bound on the number of such residue rings, and Algorithm~\ref{alg:minweight} summarizes the resulting computational procedure.
This reduction should be understood as a structural reduction rather than a complexity-theoretic improvement: it does not accelerate the computation for a specified $q$, but amortizes it over the infinite family $\{C_G(q)\}_{q\geq 1}$.

Greene’s theorem \cite{Greene1976} (see Theorem~\ref{thm-Greene}) provides a transformation formula relating a weight enumerator of a code over a field and the Tutte polynomial of the matroid associated with the code.
It has applications to the MacWilliams identity (see \cite{JP2013}).
We also establish an analog of Greene’s theorem that provides a transformation formula between the corresponding weight enumerator and the Tutte quasi-polynomial in the sense of Br\"and\'en and Moci \cite{BM2014}, both arising from an integer matrix.

Finally, we compute the number of codewords with all nonzero entries of the codes derived from the matroids $N_k$ and $Z_k$, which is essentially equivalent to computing the characteristic quasi-polynomials of the associated hyperplane arrangements.
Using coding-theoretic methods, we detect congruence conditions by employing the complete orthogonal system of additive characters
modulo $q$, thereby reducing the enumeration of codewords whose entries are all nonzero to the evaluation of certain character sums.

The organization of this paper is as follows.
In Section \ref{sec-preliminaries}, we recall the definitions and known results of the weight enumerator and the characteristic quasi-polynomial.
We also present detailed calculations for specific examples that motivated this research.
In Section \ref{sec-quasi-poly-weight-dist}, we prove the main theorem of this paper using an explicit formula for the weight enumerator as a quasi-polynomial.
In Section \ref{sec-Greene}, we give a transformation formula between the Tutte quasi-polynomial and the weight enumerator.
In Section \ref{sec-NkZk}, we compute the number of full-support codewords for the codes related to $N_k$ and $Z_k$.
Finally, in Section \ref{sec-examples}, notable examples are listed to help the reader understand the results of this paper.

\section{Preliminaries}\label{sec-preliminaries}
Let $q$ be a positive integer.
For any nonnegative integer $a$, we define $\mathbb{Z}_{>a} \coloneq \{z\in\mathbb{Z}\mid z>a\}$, $\mathbb{Z}_{\geq a} \coloneq \{z\in\mathbb{Z}\mid z\geq a\}$, and $[z] \coloneq \{1,2,\dots,z\}$ for any $z\in\mathbb{Z}_{>0}$.
For $q\in\mathbb{Z}_{>0}$, define $\mathbb{Z}_q \coloneq \mathbb{Z}/q\mathbb{Z}$.
Then $\mathbb{Z}_q$ and $\mathbb{Z}_q^n$ are naturally $\mathbb{Z}$-modules.

In this paper, let
\begin{align*}
E \coloneq \{1,2,\dots,n\}.
\end{align*}
Let $\abs{X}$ be the cardinality of a set $X$.
When a group $X$ is a subgroup of a group $Y$, we write $X \leq Y$.

\subsection{Codes over integer residue rings}
In this subsection, we review the basic facts concerning linear codes over integer residue rings.
Although classical coding theory is often developed over finite fields, linear codes over integer residue rings have also been studied since the early development of coding theory over rings.
Blake introduced and studied codes over integer residue rings, including natural analogs of Hamming, Reed--Solomon, and BCH codes over such rings \cite{Blake1975}.
A later major motivation for studying codes over finite rings is the discovery that the Kerdock code, the Preparata code, and several related nonlinear binary codes can be understood as linear codes over $\mathbb Z_4$ through the Gray map \cite{HKCSS94}.
Therefore, even when one is interested in binary codes, allowing coefficient rings such as $\mathbb{Z}_{q}$ provides a useful algebraic framework.

A $\mathbb{Z}_{q}$-\emph{linear code} of length $n$ is a $\mathbb{Z}_{q}$-submodule $C\le \mathbb{Z}_{q}^{n}$ of $\mathbb{Z}_{q}^{n}$.
A \emph{generator matrix} of $C$ is a matrix $\overline{G}\in \Mat_{k\times n}(\mathbb{Z}_{q})$ whose row vectors generate $C$ as a $\mathbb{Z}_{q}$-module, i.e., 
\begin{align*}
C=\{u\overline{G}\mid u=(u_1,\dots,u_k)\in \mathbb{Z}_{q}^{k}\} \le \mathbb{Z}_{q}^{n}.
\end{align*}
Here the row vectors of $\overline{G}$ are not required to form a basis; they are only required to generate $C$.
Since every element of $\mathbb{Z}_{q}$ can be represented by an integer, a generator matrix over $\mathbb{Z}_{q}$ can always be expressed as the residue modulo $q$ of an integer matrix $G\in \Mat_{k\times n}(\mathbb{Z})$.
With a slight abuse of terminology, we also call such an integral matrix $G$ a \emph{generator matrix} of $C\le \mathbb{Z}_{q}^{n}$.
Integral generator matrices appear naturally when the construction is specified before the residue ring is chosen.
This happens, for example, for incidence matrices of graphs and integral representations of hyperplane arrangements.
Conversely, a generator matrix over a fixed ring $\mathbb{Z}_{q}$ can be chosen as an integral matrix by choosing integer representatives of its entries.
Such a choice of a generator matrix is not unique, and different choices of the same code over $\mathbb{Z}_{q}$ may define different families over other residue classes (see Section~\ref{sec-ex-dependence}).
Thus, the periodicity studied below is associated with the chosen integral matrix, not merely with a single code over one fixed residue ring.
In particular, when $q$ is a prime power, the residue ring $\mathbb{Z}_q$ is a finite chain ring.
Linear codes over finite chain rings have been extensively studied (see, for example, \cite{NortonSalageanStructure,NortonSalageanDistance,NS09}).

For $c=(c_1,\dots,c_n)\in \mathbb{Z}_{q}^{n}$, its \emph{support} is defined by $\supp(c)\coloneq\{j\in [n]\mid c_j\neq 0\}$.
The \emph{Hamming weight}, or simply the \emph{weight}, of $c$ is $\wt(c) \coloneq \abs{\supp(c)}$.
\begin{defi}
  For a $\mathbb{Z}_{q}$-linear code $C\leq\mathbb{Z}_{q}^{n}$, the \emph{minimum weight} $d(C)$ of $C$ is defined by $d(C)\coloneq\min\{\wt(c)\mid c\in C\setminus \{0\}\}$ if $C\neq\{0\}$ and $d(C)\coloneq+\infty$ if $C=\{0\}$.
\end{defi}
The \emph{weight distribution} of $C\leq\mathbb{Z}_{q}^{n}$ is defined by
\begin{align*}
A_{C,i} \coloneq \abs{\{ c \in C \mid \wt(c) = i \}}
\qquad (0\leq i\leq n).
\end{align*}
The \emph{Hamming weight enumerator} of $C$ is defined by
\begin{align*}
W_C(x,y)\coloneq\sum_{c\in C}x^{n-\wt(c)}y^{\wt(c)}=\sum_{i=0}^{n}A_{C,i}\,x^{n-i}y^i.
\end{align*}
Computing the minimum weight and the weight distribution is difficult in general.
Indeed, even for binary linear codes, it is known that computing the minimum weight is NP-hard and that the corresponding decision problem is NP-complete \cite{Vardy1997}.
Thus, structural methods that reduce these invariants to computations over smaller rings or describe their dependence on $q$ are important.

We next review the dual code of $C\leq \mathbb{Z}_{q}^{n}$, defined by $C^{\perp}\coloneq\{v\in \mathbb{Z}_{q}^{n}\mid \langle c,v\rangle=\sum_{j=1}^{n}c_jv_j=0\text{ for all }c\in C\}$.
The relation between the weight enumerator of a linear code and that of its dual code is given by the MacWilliams identity.
For a linear code $C$ over a finite field $\mathbb{F}_{q}$, the MacWilliams identity takes the form
\begin{align*}
W_{C^{\perp}}(x,y)=\frac{1}{\abs{C}}W_C(x+(q-1)y,x-y).
\end{align*}
More generally, since $\mathbb{Z}_q$ is a finite Frobenius ring,
the Hamming weight enumerators of a $\mathbb{Z}_q$-linear code and its dual also satisfy a MacWilliams identity.
Finite Frobenius rings constitute a natural class of finite rings for which MacWilliams identities hold (see \cite{Wood1999}).
The MacWilliams identity relates a code and its dual for a fixed $q$.
By contrast, the quasi-polynomial approach of this paper fixes an integral generator matrix and studies how the weight enumerator varies as $q$ changes.

For the remainder of this paper, fix an integer matrix $G\in \operatorname{Mat}_{k\times n}(\mathbb{Z})$.
Throughout this paper, we assume that $G$ has no zero columns.
For each $q\in \mathbb{Z}_{>0}$, viewing $G$ modulo $q$ gives a matrix over $\mathbb{Z}_{q}$.
Thus, $G$ defines a $\mathbb{Z}_{q}$-linear code
\begin{align*}
C_G(q) \coloneq \left\{uG\,\middle|\,u=(u_1,\dots,u_k)\in\mathbb{Z}_q^k\right\}\leq\mathbb{Z}_{q}^{n}.
\end{align*}
Here the product $uG$ is computed modulo $q$.
In this way, by varying $q$, a family of codes $\{C_{G}(q)\}_{q\in \mathbb{Z}_{> 0}}$ over the integer residue rings $\mathbb{Z}_{q}$ is obtained from a single integer matrix $G$.

For this family, we write the weight distribution of $C_G(q)$ as
\begin{align*}
A_{G,i}(q)\coloneq\abs{\{ c \in C_G(q) \mid \operatorname{wt}(c)=i \}}
\end{align*}
for any $i\in\{0,1,\dots,n\}$.
We also write its weight enumerator as
\begin{align*}
W_{G}(x,y;q)&\coloneq W_{C_G(q)}(x,y)=\sum_{i=0}^n A_{G,i}(q)x^{n-i}y^i.
\end{align*}
The weight enumerator $W_{G}(x,y;q)$ is invariant under elementary row operations, column permutations, and unit column multiplications for $G$ over $\mathbb{Z}$.
Furthermore, we write the minimum weight of $C_G(q)$ as
\begin{align*}
  d_q \coloneq d(C_G(q)) =
        \begin{cases}
            +\infty&\ \text{if}\quad C_{G}(q) = \{ 0 \},\\
            \min\{i\in\mathbb{Z}_{>0}\mid A_{G,i}(q)\neq 0\} &\ \text{if}\quad C_{G}(q) \neq \{ 0 \}.
        \end{cases}
\end{align*}
Note that $d_1$ is the minimum weight for a code consisting of only the zero codeword.
\begin{rem}\label{rem-min-weight-distance}
The minimum value of $\abs{\{i\mid c_i\neq c^{\prime}_i\}}$ over all distinct $(c_1,\dots,c_n),(c^{\prime}_1,\dots,c^{\prime}_n)\in C_G(q)$ is called the minimum Hamming distance of $C_G(q)$.
For ring linear codes, it is well known that the minimum weight and the minimum distance coincide (see, for example, \cite[Proposition~5.3]{JP2013}).
The symbol $d$ is commonly used for the minimum distance, but in this paper, we use it as a symbol for minimum weight.
\end{rem}

\begin{rem}\label{rem-reduct-min-dist}
    Let $q,m\in\mathbb{Z}_{>0}$ with $m\mid q$, and let $\pi \colon \mathbb{Z}_q\to\mathbb{Z}_m$ be the natural reduction map, extended coordinatewise to $\pi \colon \mathbb{Z}_q^n\to\mathbb{Z}_m^n$.
    For a code $C\le\mathbb{Z}_q^n$, the image $\pi(C)\le\mathbb{Z}_m^n$ is a code.
    For any $c\in C$, we have $\frac{q}{m}\cdot c\in C$
    and $\operatorname{supp}(\frac{q}{m}\cdot c)=\operatorname{supp}(\pi(c))$.
    Hence, every weight occurring in $\pi(C)$ also occurs in $C$, i.e., $  \{ \operatorname{wt}(c) : c \in C \setminus \{ 0 \} \} \supseteq \{ \operatorname{wt}(\bar{c}) : \bar{c} \in \pi(C) \setminus \{ 0 \} \}$.
    Taking the minimum yields $d_q\le d_m$.
\end{rem}

\subsection{Motivating examples}\label{subsec-motivating-example}

In this subsection, we present calculations for concrete examples that serve as the motivation for this study.
In particular, by analyzing these examples using the language of hyperplane arrangements, we establish a bridge to the relationship with the characteristic quasi-polynomials introduced in the next subsection.

Let $g^1,g^2,\dots,g^n$ be the column vectors of $G$, i.e., $G=\left(g^1\ g^2\ \cdots\ g^n\right)$.
Since codewords are described as $c=uG=\left(ug^1,ug^2,\dots,ug^n\right)$, the number of codewords with $c_j=0$ is $\abs{\{c\in C_G(q)\mid c_j=0\}}$ $=\abs{\{uG\mid u\in\mathbb{Z}_q^k,\ ug^j=0\}}$.
Also, we have
\begin{align*}
    \abs{\{ c \in C_G(q) \mid \operatorname{supp}(c) = E \setminus J \}}=\abs{\bigcap_{j\notin J}\{uG\mid u\in\mathbb{Z}_q^k,\ ug^j\neq 0\}\cap \bigcap_{j\in J}\{uG\mid u\in\mathbb{Z}_q^k,\ ug^j=0\}}.
\end{align*}
The coefficients $A_{G,i}(q)$ of the weight enumerator are also described as
\begin{align}\label{eq-AGi-def-sum}
    A_{G,i}(q)=\sum_{\substack{J \subseteq E \\ |J| = n - i}}|\{ c \in C_G(q) \mid \operatorname{supp}(c) = E \setminus J \}|.
\end{align}
Therefore, the sets $\{uG\mid u\in\mathbb{Z}_q^k,\ ug^j=0\}$ are important for calculating the weight enumerator.

If we consider counting the corresponding elements of $\mathbb{Z}_q^k$ instead of the subsets $\{uG\mid u\in\mathbb{Z}_q^k,\ ug^j=0\}$ of $C \subset \mathbb{Z}_q^n$, the notion of a ``hyperplane arrangement'' arises naturally.
The matrix $G$ defines the hyperplanes\footnote{We call $H_{j}(q)$ a hyperplane in $\mathbb{Z}_q^{k}$ by a slight abuse of terminology.} $H_{j}(q) \coloneq \{u\in\mathbb{Z}_q^k\mid u g^j=0\}$ for $j\in E$.
We regard $\mathcal{A}_q \coloneq \{H_{j}(q)\mid j\in E\}$ as a multiarrangement in $\mathbb{Z}_q^k$, where multiplicities are determined by the index set $E$.
The $i$-th entry of the codeword $c=uG=\left(ug^1,ug^2,\dots,ug^n\right)$ being zero corresponds to $u \in \mathbb{Z}_{q}^{k}$ being contained in the hyperplane $H_i(q)$.

\begin{example}\label{example-B2-20/02}
\begin{figure}[t]
  \centering
  \begin{minipage}{0.45\textwidth}
    \centering
    \begin{tikzpicture}[scale=0.8]
      \draw[step=1,black] (0,0) grid (4,4);
      \draw[red, ultra thick] (0,0) -- (4,0);
      \draw[red, ultra thick] (0,0) -- (0,4);
      \draw[red, ultra thick] (0,0) -- (4,4);
      \draw[red, ultra thick] (1,4) -- (4,1);
      \draw[red, ultra thick] (-0.5,0.5) -- (0.5,-0.5);
      \draw[fill=black] (0,0) circle (0.25);
      \node[below left] at (-0.1,-0.1) {$0$};
    \end{tikzpicture}
    \caption{Hyperplanes and intersection points of the arrangement in Example \ref{example-B2-20/02} \textup{(1)} when $q=5$}
  \end{minipage}
  \hfill
  \begin{minipage}{0.45\textwidth}
    \centering
    \begin{tikzpicture}[scale=0.7]
      \draw[step=1,black] (0,0) grid (5,5);
      \draw[red, ultra thick] (0,0) -- (5,0);
      \draw[red, ultra thick] (0,0) -- (0,5);
      \draw[red, ultra thick] (0,0) -- (5,5);
      \draw[red, ultra thick] (1,5) -- (5,1);
      \draw[red, ultra thick] (-0.5,0.5) -- (0.5,-0.5);
      \draw[fill=black] (3,3) circle (0.25);
      \draw[fill=black] (0,0) circle (0.25);
      \node[below left] at (-0.1,-0.1) {$0$};
      \node[below] at (3,0) {$\frac{q}{2}$};
      \node[left] at (0,3) {$\frac{q}{2}$};
    \end{tikzpicture}
    \caption{Hyperplanes and intersection points of the arrangement in Example \ref{example-B2-20/02} \textup{(1)} when $q=6$}
  \end{minipage}
\end{figure}
    \textup{(1)} Let $k=2$, $n=4$, and $G=\begin{pmatrix} 1&0&1&1\\ 0&1&1&-1 \end{pmatrix}$.
    Let $q\in\mathbb{Z}_{>0}$.
    For any $u=(u_1,u_2)\in \mathbb{Z}_{q}^k$, we have $uG=(u_1,u_2,u_1+u_2,u_1-u_2)$.
    Thus, the hyperplanes in $\mathcal{A}_q$ are $H_1(q)=\{u\in\mathbb{Z}_q^k\mid u_1=0\}$, $H_2(q)=\{u\in\mathbb{Z}_q^k\mid u_2=0\}$, $H_3(q)=\{u\in\mathbb{Z}_q^k\mid u_1+u_2=0\}$, and $H_4(q)=\{u\in\mathbb{Z}_q^k\mid u_1-u_2=0\}$.
    In this case, $uG=u^{\prime}G$ is equivalent to $u=u^{\prime}$ for any $u,u^{\prime}\in\mathbb{Z}_q^k$.
    Therefore, $A_{G,i}(q)$ is the number of elements $u\in\mathbb{Z}_q^k$ contained in exactly $n-i$ hyperplanes in $\mathcal{A}_q$.
    From now on, we calculate $A_{G,i}(q)$ separately for when $q$ is even and when $q$ is odd.

    Let $q$ be odd.
    Since $A_{G,4}(q)$ is the number of elements in $\mathbb{Z}_{q}^k\setminus\bigcup_{i\in E}H_i(q)$, we have $A_{G,4}(q)=q^2-4q+3=(q-1)(q-3)$ by the inclusion-exclusion principle.
    Next, $A_{G,3}(q)$ is the number of elements contained in exactly one hyperplane.
    Since the origin is contained in four hyperplanes and each hyperplane contains $q-1$ elements except for the origin, we have $A_{G,3}(q)=4(q-1)$.
    There are no elements contained in exactly two or exactly three hyperplanes.
    Thus, we have $A_{G,2}(q)=A_{G,1}(q)=0$.
    Finally, $A_{G,0}(q)=1$ since the only element contained in all hyperplanes is the origin.
    Therefore, we have
    \begin{align*}
        W_G(x,y;q)=x^4+4(q-1)xy^3+(q-1)(q-3)y^4.
    \end{align*}

    Let $q$ be even.
    In this case, $(q/2,q/2)$ is contained in $H_3(q)$ and $H_4(q)$.
    By counting the number of elements in the complement of the union of all hyperplanes, we have $A_{G,4}(q)=q^2-4q+4=(q-2)^2$.
    The number of elements contained in exactly one hyperplane is $A_{G,3}(q)=4(q-1)-2=4q-6$.
    Unlike the case where $q$ is odd, there exists only one element contained in exactly two hyperplanes.
    Thus, we have $A_{G,2}(q)=1$.
    Since $A_{G,1}(q)=0$ and $A_{G,0}(q)=1$, we have
    \begin{align*}
        W_G(x,y;q)=x^4+x^2y^2+(4q-6)xy^3+(q-2)^2y^4.
    \end{align*}

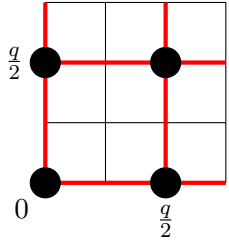
\begin{figure}[t]
  \centering
  \begin{tikzpicture}[scale=0.8]
    \draw[step=1,black] (0,0) grid (3,3);
    \draw[red, ultra thick] (0,0) -- (3,0);
    \draw[red, ultra thick] (0,0) -- (0,3);
    \draw[red, ultra thick] (2,0) -- (2,3);
    \draw[red, ultra thick] (0,2) -- (3,2);
    \foreach \x/\y in {0/0, 2/0, 0/2, 2/2}
      \draw[fill=black] (\x,\y) circle (0.25);
    \node[below left] at (-0.1,-0.1) {$0$};
    \node[below] at (2,-0.2) {$\frac{q}{2}$};
    \node[left] at (-0.2,2) {$\frac{q}{2}$};
  \end{tikzpicture}
  \caption{Hyperplanes and intersection points of the arrangement in Example \ref{example-B2-20/02} \textup{(2)} when $q=4$}
\end{figure}
    \textup{(2)} Let $k=2$, $n=2$, and $G=\begin{pmatrix} 2&0\\ 0&2 \end{pmatrix}$.
    Let $q\in\mathbb{Z}_{>0}$.
    For any $u=(u_1,u_2)\in \mathbb{Z}_{q}^k$, we have $uG=(2u_1,2u_2)$.
    The hyperplanes in $\mathcal{A}_q$ are $H_1(q)=\{u\in\mathbb{Z}_q^k\mid 2u_1=0\}$ and $H_2(q)=\{u\in\mathbb{Z}_q^k\mid 2u_2=0\}$.
    
    If $q$ is odd, then $G$ can be considered as the identity matrix.
    Thus, we have $W_G(x,y;q)=((q-1)y+x)^2=x^2+2(q-1)xy+(q-1)^2y^2$.

    Let $q$ be even.
    In this case, $uG=u^{\prime}G$ does not imply $u=u^{\prime}$.
    For $c\in C_G(q)$, there are four elements $u \in \mathbb{Z}_{q}^{k}$ satisfying $uG = c$.
    In addition, $(0,0)$, $(q/2,0)$, $(0,q/2)$, and $(q/2,q/2)$ are contained in exactly two hyperplanes.
    Therefore, we have
    \begin{align*}
        &A_{G,0}(q)=\frac{4}{4}=1,\ 
        A_{G,1}(q)=\frac{4(q-2)}{4}=q-2,\ 
        A_{G,2}(q)=\frac{q^2-4q+4}{4}=\left(\frac{q}{2}-1\right)^2\\
        &W_G(x,y;q)=x^2+(q-2)xy+\left(\frac{q} {2}-1\right)^2y^2.
    \end{align*}
\end{example}
As seen in these examples, it is observed that the weight enumerators and minimum weights exhibit a certain periodicity.
In Example \ref{example-B2-20/02} \textup{(2)}, when $q=2$, the coefficients of $W_G(x,y;q)$ in $xy$ and $y^2$ are zero, which means that no codewords of weight $1$ or $2$ appear.
On the other hand, when $q$ is even and $q>2$, codewords of weight $1$ or weight $2$ appear, and the minimum weight is $1$.
Therefore, in some cases, the periodicity of the minimum weight also requires excluding special values of $q$.
Furthermore, in Example \ref{example-B2-20/02} \textup{(2)}, a polynomial in $q$ with rational coefficients appears.

\subsection{Characteristic quasi-polynomials}\label{subsec-Char-quasi-poly}
As described in the examples in Section \ref{subsec-motivating-example}, the weight enumerator can be computed by counting the number of hyperplanes on which each point of $\mathbb{Z}_q^k$ lies.
This approach is well-suited to the characteristic quasi-polynomials introduced by Kamiya, Takemura, and Terao~\cite{KTT08,KTT11}. In this paper, we apply the theory of characteristic quasi-polynomials to weight enumerators and discuss the periodicity of minimum weights.

Let $R$ and $R^{\prime}$ be commutative rings with $R\subseteq R^{\prime}$. 
A function $f\colon\mathbb{Z}_{>0}\rightarrow R$ is called a \emph{quasi-polynomial} over $R^{\prime}$ if there exists a positive integer $\rho\in\mathbb{Z}_{>0}$ and polynomials $f^1(t),f^2(t)$, $\ldots,f^{\rho}(t)\in R^{\prime}[t]$ such that for $q\in\mathbb{Z}_{>0}$,
\begin{align*}
f(q)=
\begin{cases}
f^1(q)\quad&q\equiv 1\  \mod\ \rho,\\
f^2(q)\quad&q\equiv 2\  \mod\ \rho,\\
\quad\vdots&\quad\vdots\\
f^{\rho}(q)\quad&q\equiv \rho\ \mod\ \rho.\\
\end{cases}
\end{align*}
The number $\rho$ is called a \emph{period}, and the polynomial $f^m(t)$ ($m\in[\rho]$) is called the $m$-\emph{constituent} of the quasi-polynomial $f$.
The minimum value of the periods of the quasi-polynomial $f$ is called the \emph{minimum period} of $f$.
In addition, we call $f$ a monic quasi-polynomial if $f^1(t),f^2(t),\dots,f^{\rho}(t)$ are monic polynomials.
A quasi-polynomial $f$ with a period $\rho$ is said to have the \emph{gcd property with respect to $\rho$} if the $m$-constituents $f^1(t),f^2(t),\dots,f^{\rho}(t)$ depend on $m$ only through $\gcd(\rho,m)$, i.e., $f^a(t)=f^b(t)$ if $\gcd(a,\rho)=\gcd(b,\rho)$.

Let $G\in\operatorname{Mat}_{k\times n}(\mathbb{Z})$ and let $J$ be a subset of $E$ with $J\neq\emptyset$.
We write $J=\{j_{1},j_{2},\dots, j_{\abs{J}}\}$ and set $G_{J} \coloneq (g^{j_1}\ g^{j_2}\ \cdots\ g^{j_{\abs{J}}})\in \operatorname{Mat}_{k\times \abs{J}}(\mathbb{Z})$, i.e., $G_{J}$ is the submatrix of $G$ consisting of columns indexed by $J$.
An integer square matrix $P$ is said to be unimodular if $\det(P)\in\{1,-1\}$.
Note that the inverse matrix of the unimodular matrix is an integer matrix.
From the theory of the Smith normal form, there exist unimodular matrices $P\in \operatorname{Mat}_{k\times k}(\mathbb{Z})$, $Q\in \operatorname{Mat}_{\abs{J}\times \abs{J}}(\mathbb{Z})$, and positive integers $e_{1,J},e_{2,J},\dots,e_{r(J),J}\in\mathbb{Z}_{>0}$ such that
\begin{align*}
PG_J Q=\operatorname{diag}(e_{1, J},e_{2, J},\dots,e_{r(J), J},0,\dots,0)=
\begin{pmatrix}
  e_{1,J} &0 &\cdots&\cdots&\cdots&\cdots&0\\
  0& e_{2, J} &\ddots& & & &\vdots\\
  \vdots&\ddots& \ddots &\ddots& & &\vdots \\
  \vdots& &\ddots& e_{r(J), J} &\ddots& &\vdots \\
  \vdots& & &\ddots& 0 &\ddots&\vdots\\
  \vdots& & & &\ddots& \ddots &0 \\
  0&\cdots&\cdots&\cdots&\cdots&0& 0
\end{pmatrix}
\end{align*}
and $e_{1,J}|e_{2,J}|\cdots|e_{r(J),J}$.
Such integers $e_{1,J},e_{2,J},\dots,e_{r(J),J}$ are unique and are called \emph{elementary divisors} of $G_J$.
In addition, $r(J)$ is called the rank of $G_J$.
Note that $r(J)\geq 1$ in this case, since $G$ does not contain zero column vectors.
If $J\subseteq J^{\prime}$, then we have $r(J)\leq r(J^{\prime})$ since the set of columns indexed by $J$ is included in the set of columns indexed by $J^{\prime}$.
In the case where $J=\emptyset$, we define $r(\emptyset)\coloneq 0$.
Then, the \emph{lcm period} $\rho_0$ of $G$ is defined as the lcm of the last elementary divisors $e_{r(J),J}$ of $G_J$, that is,
\begin{align*}
\rho_0 \coloneq \operatorname{lcm}\left(e_{r(J),J}\,\middle|\,\emptyset\neq J\subseteq E\right).
\end{align*}
As $q\in\mathbb{Z}_{>0}$ varies, the cardinality $\abs{\mathbb{Z}_q^k\setminus \bigcup_{j\in E}H_j(q)}$ of the complement of the union of hyperplanes can be regarded as a function of $q$. 

\begin{thm}[Kamiya, Takemura, and Terao~\cite{KTT08}]\label{thm-charquasi-KTT}
The function $\chi_{G}^{\mathrm{quasi}} \colon \mathbb{Z}_{> 0} \to \mathbb{Z}$ defined as $\chi^{\mathrm{quasi}}_G(q)=\abs{\mathbb{Z}_q^k\setminus \bigcup_{j\in E}H_j(q)}$ is a monic quasi-polynomial over $\mathbb{Z}$.
Moreover, $\rho_{0}$ is a period of $\chi_{G}^{\mathrm{quasi}}$, and $\chi_{G}^{\mathrm{quasi}}$ has the gcd property with respect to $\rho_{0}$.
\end{thm}
The quasi-polynomial in Theorem \ref{thm-charquasi-KTT} is called the \emph{characteristic quasi-polynomial}.
Theorem \ref{thm-charquasi-KTT} states that $\rho_0$ is a period of the characteristic quasi-polynomial.
Although Theorem \ref{thm-charquasi-KTT} does not state whether $\rho_0$ is the minimum period, Higashitani, Tran, and Yoshinaga~\cite{HTY2023} have proven that $\rho_0$ is indeed the minimum period\footnote{The claim in the paper~\cite{HTY2023} is that “in the central case, the minimum period of the characteristic quasi-polynomial coincides with the lcm period.” Here, the central case refers to the situation where all hyperplanes pass through the origin. Several references, including the paper~\cite{HTY2023}, discuss the minimum period in non-central cases.}.
\begin{thm}[Higashitani, Tran, and Yoshinaga~\cite{HTY2023}]\label{thm-charquasi-HTY}
    The minimum period of a characteristic quasi-polynomial coincides with the lcm period $\rho_0$.
\end{thm}

\section{Quasi-polynomial of weight distribution}\label{sec-quasi-poly-weight-dist}
In this section, we discuss the periodicity of the minimum weights and give a necessary and sufficient condition under which the minimum weights are periodic.
\subsection{Constituents of the weight enumerator}\label{sec-min}
Let $G\in\operatorname{Mat}_{k\times n}(\mathbb{Z})$.
For $J\subseteq E$, we define
\begin{align}\label{eq-def-int-planes}
H_{J}(q) \coloneq \bigcap_{j\in J} H_{j}(q).
\end{align}
We conventionally define $\prod_{i=1}^{0}z_i \coloneq 1$.

\begin{lem}\label{lem-num-ele-H}
For $J\subseteq E$, we have
\begin{align}
\abs{H_{J}(q)}=\left(\prod_{j=1}^{r(J)}\operatorname{gcd}(q,e_{j,J})\right)q^{k-r(J)}.
\end{align}
\end{lem}
\begin{proof}
If $J=\emptyset$, then we have $\abs{H_{J}(q)}=\abs{\mathbb{Z}_q^k}=q^k$.
Let $J\neq\emptyset$.
For any $u\in\mathbb{Z}_q^k$, there exist unimodular matrices $P\in \operatorname{Mat}_{k\times k}(\mathbb{Z})$, $Q\in \operatorname{Mat}_{\abs{J}\times \abs{J}}(\mathbb{Z})$ such that
\begin{align*}
uG_J=0\ \Leftrightarrow\ uP^{-1}PG_JQ=0\ \Leftrightarrow\ \left(uP^{-1}\right)\operatorname{diag}(e_{1, J},e_{2, J},\dots,e_{r(J), J},0,\dots,0)=0.
\end{align*}
Since the $\mathbb{Z}$-homomorphism $\phi_{P^{-1}} \colon \mathbb{Z}_q^k\rightarrow\mathbb{Z}_q^k$ defined as $\phi_{P^{-1}}(u)=uP^{-1}$ for $u\in\mathbb{Z}_q^k$ is an isomorphism,
\begin{align*}
\abs{H_{J}(q)} &=\abs{\{u\in\mathbb{Z}_q^k\mid uG_J=0\}}\\
&=\abs{\{u^{\prime}=(u^{\prime}_1,\dots,u^{\prime}_k)\in\mathbb{Z}_q^k\mid (u^{\prime}_1e_{1,J},u^{\prime}_2e_{2,J},\dots,u^{\prime}_{r(J)}e_{r(J),J},0,\dots,0)=0\}}\\
&=\left(\prod_{j=1}^{r(J)}\operatorname{gcd}(q,e_{j,J})\right)q^{k-r(J)}.
\end{align*}
The last equality follows from $\abs{\{ y \in \mathbb{Z}_q \mid ye_{j,J} = 0 \}} = \abs{ \left\{ \frac{aq}{\operatorname{gcd}(q,e_{j,J})}\mathrel{}\middle |\mathrel{} a\in\{0,1,\dots,\operatorname{gcd}(q,e_{j,J})-1\} \right\}}$ for any $j\in [r(J)]$.
\end{proof}

\begin{lem}\label{lem-crit-thm}
For any $J\subseteq E$,
\begin{align}\label{eq-supp-HJ}
    \abs{\left\{ c \in C_{G}(q) \mid \operatorname{supp}(c) = E \setminus J \right\}} \cdot \abs{H_E(q)} = \abs{H_{J}(q) \setminus \bigcup_{j \in E \setminus J} H_{j}(q)}.
\end{align}
\end{lem}
\begin{proof}
Let $\varphi_{G, q} \colon \mathbb{Z}_{q}^{k} \to C_G(q)$ be the $\mathbb{Z}$-homomorphism defined as $\varphi_{G, q}(u) \coloneq uG=(u\cdot g^1,u\cdot g^2,\dots,u\cdot g^n)$.
Clearly, we see that $\varphi_{G, q}$ is surjective, that is, $\operatorname{Im}\varphi_{G, q}=C_G(q)$.

Here, for any $c\in C_G(q)$ with $\operatorname{supp}(c)=E\setminus J$, there exists $u\in \mathbb{Z}_q^k$ with $\operatorname{supp}(uG)=E\setminus J$ such that $c=\varphi_{G, q}(u)$.
In addition, $\varphi_{G, q}^{-1}(c)=u+\operatorname{Ker}\varphi_{G, q}$.
Therefore, we have
\begin{multline*}
    \abs{\{c\in C_G(q)\mid \operatorname{supp}(c)=E\setminus J\}} \cdot \abs{\operatorname{Ker}\varphi_{G, q}}
    = \abs{\{u\in\mathbb{Z}_q^k\mid \operatorname{supp}(uG)=E\setminus J\}}\\
    = \abs{\{u\in\mathbb{Z}_q^k\mid ug^i = 0\ \text{if}\ i\in J,\ \text{and}\ ug^j \neq 0\ \text{if}\ j\in E\setminus J\}}
    =\abs{H_{J}(q) \setminus \bigcup_{j \in E \setminus J} H_{j}(q)}.
\end{multline*}
Since $\operatorname{Ker}\varphi_{G, q}=H_E(q)$, we obtain Equation \eqref{eq-supp-HJ}.
\end{proof}

Applying Lemma \ref{lem-crit-thm} when $J=\emptyset$, the number of codewords of weight $n$ is given by the characteristic quasi-polynomial divided by $\abs{H_E(q)}$, as follows.
\begin{cor}\label{cor-weight=n}
    We have 
    \begin{align*}
    A_{G,n}(q)=\frac{\abs{\mathbb{Z}_q^k\setminus \bigcup_{j \in E} H_{j}(q)}}{\abs{H_E(q)}}=\frac{\chi^{\mathrm{quasi}}_G(q)}{\abs{H_E(q)}}.    
    \end{align*}
\end{cor}

We now describe $A_{G,i}(q)$ for $i\in\{0,1,\dots,n\}$ using elementary divisors $e_{1,J}|e_{2,J}|\cdots|e_{r(J),J}$ and provide an explicit formula for $A_{G,i}(q)$ as a quasi-polynomial.
By the inclusion–exclusion principle,
\begin{align*}
   \abs{H_{J}(q) \setminus \bigcup_{j \in E \setminus J} H_{j}(q)} &= \abs{H_{J}(q)} + \sum_{j=0}^{n - \abs{J} -1} (-1)^{n - \abs{J} - j} \sum_{\substack{J^{\prime} \subseteq E \setminus J \\ \abs{J^{\prime}} = n - \abs{J} - j}} \abs{H_{J}(q) \cap H_{J^{\prime}}(q)} \\
   &= \sum_{j=0}^{n - \abs{J}} (-1)^{n - \abs{J} - j} \sum_{\substack{J^{\prime} \subseteq E \setminus J \\ \abs{J^{\prime}} = n - \abs{J} - j}} \abs{H_{J\cup J^{\prime}}(q)}.
\end{align*}
By Lemma \ref{lem-num-ele-H}, 
\begin{align*}
    \frac{\abs{H_{J\cup J^{\prime}}(q)}}{\abs{H_E(q)}}&=\frac{\prod_{\ell=1}^{r(J\cup J^{\prime})}\gcd(q,e_{\ell, J\cup J^{\prime}}) q^{k - r(J\cup J^{\prime})}}{\prod_{\ell=1}^{r(E)}\gcd(q,e_{\ell, E}) q^{k - r(E)}}
    =\frac{\prod_{\ell=1}^{r(J\cup J^{\prime})}\gcd(q,e_{\ell, J\cup J^{\prime}}) }{\prod_{\ell=1}^{r(E)}\gcd(q,e_{\ell, E})}q^{r(E) - r(J\cup J^{\prime})}
\end{align*}
for any $J^{\prime}\subseteq E \setminus J$.
Therefore, by Lemma \ref{lem-crit-thm}, we have
\begin{align}
    A_{G,i}(q) &= \sum_{\substack{J \subseteq E \\ \abs{J} = n - i}}\abs{\{ c \in C_G(q) \mid \operatorname{supp}(c) = E \setminus J \}} \notag\\
    &=\frac{\sum_{\substack{J \subseteq E \\ \abs{J} = n - i}}\abs{H_{J}(q) \setminus \bigcup_{j \in E \setminus J} H_{j}(q)}}{\abs{H_E(q)}} \notag\\
    &= \sum_{\substack{J \subseteq E \\ \abs{J} = n - i}} \sum_{j=0}^{i} (-1)^{i - j} \sum_{\substack{J^{\prime} \subseteq E \setminus J \\ \abs{J^{\prime}} = i-j}} \frac{\abs{H_{J\cup J^{\prime}}(q)}}{\abs{H_E(q)}} \notag\\
    &= \sum_{\substack{J \subseteq E \\ \abs{J} = n - i}} \sum_{j=0}^{i} (-1)^{i-j} \sum_{\substack{J^{\prime} \subseteq E \setminus J \\ \abs{J^{\prime}} = i-j}} \frac{\prod_{\ell = 1}^{r(J\cup J^{\prime})} \gcd(q,e_{\ell, J\cup J^{\prime}})}{\prod_{\ell = 1}^{r(E)} \gcd(q,e_{\ell, E})} q^{r(E) - r(J\cup J^{\prime})} \notag\\
    &= \sum_{\substack{J \subseteq E \\ \abs{J} = n - i}}\sum_{K:J \subseteq K \subseteq E} (-1)^{\abs{K} - \abs{J}} \frac{\prod_{\ell = 1}^{r(K)} \gcd(q,e_{\ell, K})}{\prod_{\ell = 1}^{r(E)} \gcd(q,e_{\ell, E})} q^{r(E) - r(K)},\label{eq-AGi-sum}
\end{align}
where $\sum_{K:J \subseteq K \subseteq E}$ means that $K$ runs over the subsets of $E$ with $J\subseteq K$.
Here we note that $r(E) - r(K)\geq 0$ for any $K$ with $J\subseteq K\subseteq E$.
Let $\rho_0$ be the lcm period of $G$.
We define the polynomial $f_{i}^{m}(t)\in\mathbb{Q}[t]$ as
\begin{align}\label{eq-def-f_i^m}
  f_{i}^{m}(t) 
  &\coloneq \sum_{\substack{J \subseteq E \\  \abs{J} = n - i}}\sum_{K:J \subseteq K \subseteq E} (-1)^{\abs{K} - \abs{J}} \frac{\prod_{\ell = 1}^{r(K)} \gcd(m,e_{\ell, K})}{\prod_{\ell = 1}^{r(E)} \gcd(m,e_{\ell, E})} t^{r(E) - r(K)}
\end{align}
for any $m\in[\rho_0]$ and for any $i\in\{0,1,\dots,n\}$.

Now, let $m\in[\rho_0]$ and let $q\in m+\rho_0\mathbb{Z}_{\geq 0}$.
We can write $q=m+a\rho_0$ for some $a\in\mathbb{Z}_{\geq 0}$.
Since $e_{\ell,J}|\rho_0$, we have $\gcd(q,e_{\ell,J})=\gcd(m+a\rho_0,e_{\ell,J})=\gcd(m,e_{\ell,J})$ for any $J\subseteq E$.
From the above argument, we have the following.
\begin{prop}\label{prop-quasipoly-AGi}
    The number $A_{G,i}(q)$ of codewords with weight $i$ contained in $C_G(q)$ is a quasi-polynomial over $\mathbb{Q}$.
    Moreover, $\rho_0$ is a period, and $A_{G,i}(q)$ has the gcd property with respect to $\rho_0$.
More precisely, for any $m\in[\rho_0]$, we have
\begin{align*}    
    A_{G,i}(q)=f_{i}^{m}(q)\qquad \text{if}\quad q\equiv m \mod \rho_0,
\end{align*}
and $f_{i}^{m}(t)=f_{i}^{m^{\prime}}(t)$ if $\gcd(m,\rho_0)=\gcd(m^{\prime},\rho_0)$.    
\end{prop}

We next determine the minimum period of $A_{G, n}$.
By Proposition~\ref{prop-quasipoly-AGi}, $\rho_{0}$ is a period of $A_{G, n}$.
Let $\rho$ be the minimum period of $A_{G, n}$.
Then we have $\rho \mid \rho_{0}$.
For $m \in [\rho_{0}]$ , Corollary~\ref{cor-weight=n} and Lemma~\ref{lem-num-ele-H} give the identity
\begin{equation*}
  \chi_{G}^{m}(t) = \left(\prod_{\ell=1}^{r(E)}\gcd(m,e_{\ell,E})\right)
  t^{k-r(E)} f_{n}^{m}(t),
\end{equation*}
where $\chi_{G}^{m}(t)$ denotes the $m$-constituent of $\chi_{G}^{\mathrm{quasi}}$ with respect to the period $\rho_{0}$.
Since $\chi_{G}^{m}(t)$ is monic, the leading coefficient of $f_{n}^{m}(t)$ is $\displaystyle \frac{1}{\prod_{\ell=1}^{r(E)}\gcd(m,e_{\ell,E})}$.
Now, let $m, m^{\prime} \in [\rho_0]$ with $m \equiv m^{\prime} \pmod{\rho}$.
Since $\rho$ is a period of $A_{G,n}$ , we have $f_{n}^{m}(t) = f_{n}^{m^{\prime}}(t)$.
Comparing leading coefficients gives $\prod_{\ell=1}^{r(E)} \gcd(m, e_{\ell,E}) = \prod_{\ell=1}^{r(E)}\gcd(m^{\prime}, e_{\ell,E})$.
Hence, we also have $\chi_{G}^{m}(t)=\chi_{G}^{m^{\prime}}(t)$.
Therefore, $\rho$ is a period of $\chi_G^{\mathrm{quasi}}$.
Since the latter has minimum period $\rho_0$ by Theorem~\ref{thm-charquasi-HTY}, we obtain $\rho = \rho_{0}$.
Consequently, the minimum period of $A_{G,n}$ is $\rho_{0}$.

\begin{cor}\label{cor-weight-minimum-period}
    The function $W_G(x,y;-):\mathbb{Z}_{> 0}\to\mathbb{Z}[x,y]$ is a quasi-polynomial over $\mathbb{Q}[x,y]$.
    Its minimum period is $\rho_0$, and it has the gcd property with respect to $\rho_0$.
\end{cor}

In any of the examples listed in Section \ref{sec-examples}, when the coefficients $A_{G,i}(q)$ of $W_G(x,y;q)$ are regarded as polynomials in $q$, it can be seen that all degrees $0,1,\dots,r(E)$ appear.

\begin{thm}\label{thm-degreesfm-onebyone}
    Let $m\in[\rho_0]$.
    For any $s\in\{0,\dots, r(E)\}$, there exists an index $i$ with $\deg f_i^m(t)=s$.
    Moreover, for any $s\in \{1,\dots, r(E)\}$, the following inequality holds:
    \begin{equation*}
        \min \{i\in \{0,1,\dots,n\}\mid \deg f_i^m(t)=s-1\}<\min \{i\in \{0,1,\dots,n\}\mid \deg f_i^m(t)=s\}.
    \end{equation*}
\end{thm}
To prove Theorem \ref{thm-degreesfm-onebyone}, we add several definitions.
Let $s\in\{0,1,\dots,r(E)\}$.
Define $\mathcal{J}^{\prime}_s \coloneq \{J\in 2^E\mid r(E)-r(J)=s\}$, and we have $\mathcal{J}^{\prime}_s\neq\emptyset$.
Indeed, starting from the empty set, the rank increases by at most $1$ every time a new element is added, and thus there exists a subset $J$ of $E$ such that $r(J)=r(E)-s$.
Then, we define the set $\mathcal{J}_s$ as the subset of $\mathcal{J}^{\prime}_s$ whose elements have the maximum cardinality in $\mathcal{J}^{\prime}_s$.
The set $\mathcal{J}_s$ is also nonempty.
Let $\iota_{s}$ be the common cardinality of the elements of $\mathcal{J}_{s}$, that is, $\iota_s=\abs{J}$ for any $J\in\mathcal{J}_s$.
Define $i_s \coloneq n-\iota_s$.
\begin{lem}\label{lem-degreesfm-onebyone}
    Let $m\in[\rho_0]$. Then the following hold.
    \begin{itemize}
        \item[\textup{(1)}] $i_0<i_1<\cdots<i_{r(E)}$.
        \item[\textup{(2)}] $\deg f^m_{i_s}(t)=s$ for any $s\in\{0,1,\dots,r(E)\}$.
        \item[\textup{(3)}] If $\deg f^m_{i}(t)=s$, then we have $i_s\leq i$, i.e., $i_s=\min \{i\in \{0,1,\dots,n\}\mid \deg f_i^m(t)=s\}$.
    \end{itemize}
\end{lem}
\begin{proof}
    \textup{(1)}\quad We prove that $i_{s-1}<i_{s}$ for any $s\in[r(E)]$.
    Let $J_1\in\mathcal{J}_{s-1}$ and let $J_2\in\mathcal{J}_{s}$.
    Since $r(J_1)=r(E)-s+1>r(E)-s=r(J_2)$, there exists a nonempty subset $K$ of $E\setminus J_2$ such that $r(J_1)=r(J_2\cup K)$.
    If $\abs{J_1}\leq \abs{J_2}$, then $\abs{J_1}<\abs{J_2\cup K}$ and $r(E)-s+1=r(J_1)=r(J_2\cup K)$ which contradicts the definition of $\mathcal{J}_{s-1}$.
    Thus, we have $\abs{J_1} > \abs{J_2}$.
    In other words, $\iota_{s-1} > \iota_{s}$.
    Therefore, we have $i_{s-1}<i_{s}$.
    
    \textup{(2)}\quad We first prove that $r(E)-r(K)<s$ for any $K\in2^E\setminus\mathcal{J}_s$ with $\abs{K}\geq \iota_s$.
    Let $K$ be an element in $2^E\setminus\mathcal{J}_s$ with $\abs{K}\geq \iota_s$.
    If $r(K) < r(E)-s$, then there exists a nonempty subset $K^{\prime}$ of $E\setminus K$ such that $r(K\cup K^{\prime})=r(E)-s$.
    Since $\abs{K\cup K^{\prime}} > \iota_s$ and $r(E)-r(K\cup K^{\prime})=s$, this is a contradiction to the definition of $\mathcal{J}_s$.
    If $r(K) = r(E)-s$ and $\abs{K}>\iota_s$, then this also contradicts the definition of $\mathcal{J}_s$.
    If $r(K) = r(E)-s$ and $\abs{K}=\iota_s$, then we have $K\in\mathcal{J}_s$, a contradiction.
    Thus, $r(E)-r(K)<s$.
    
    Since $r(E)-r(J)=s$ and $\abs{J}-n+i_s=0$ for any $J\in\mathcal{J}_s$, by Equation \eqref{eq-def-f_i^m}, we have
    \begin{align*}
        f_{i_s}^m(t)&=\sum_{\substack{J \subseteq E \\ \abs{J} = n - i_s}}\sum_{K:J \subseteq K \subseteq E} (-1)^{\abs{K} - \abs{J}} \frac{\prod_{\ell = 1}^{r(K)} \gcd(m,e_{\ell, K})}{\prod_{\ell = 1}^{r(E)} \gcd(m,e_{\ell, E})} t^{r(E) - r(K)}\\
        &=\sum_{\substack{K \subseteq E \\ \abs{K} \geq n - i_s}} \binom{\abs{K}}{n - i_{s}}(-1)^{\abs{K} - n + i_s} \frac{\prod_{\ell = 1}^{r(K)} \gcd(m,e_{\ell, K})}{\prod_{\ell = 1}^{r(E)} \gcd(m,e_{\ell, E})} t^{r(E) - r(K)}\\
        &=\sum_{J\in\mathcal{J}_s} \frac{\prod_{\ell = 1}^{r(J)} \gcd(m,e_{\ell, J})}{\prod_{\ell = 1}^{r(E)} \gcd(m,e_{\ell, E})} t^{s} + \sum_{\substack{K \in 2^E\setminus \mathcal{J}_s \\ \abs{K} \geq n - i_s}} \binom{\abs{K}}{n - i_{s}}(-1)^{\abs{K} - n + i_s} \frac{\prod_{\ell = 1}^{r(K)} \gcd(m,e_{\ell, K})}{\prod_{\ell = 1}^{r(E)} \gcd(m,e_{\ell, E})} t^{r(E) - r(K)}.
    \end{align*}
    Here we use the fact that $\binom{\abs{K}}{n-i_s}=1$ for every $K\in\mathcal{J}_s$, since $\abs{K}=\iota_s=n-i_s$.
    Therefore, $\deg f^m_{i_s}(t)=s$ since $\displaystyle\frac{\prod_{\ell = 1}^{r(J)} \gcd(m,e_{\ell, J})}{\prod_{\ell = 1}^{r(E)} \gcd(m,e_{\ell, E})}>0$ for any $J\in\mathcal{J}_s$.

    \textup{(3)}\quad By Equation \eqref{eq-def-f_i^m}, the condition $\deg f^m_{i}(t)=s$ implies that there exists $J\in 2^E$ such that $\abs{J}\geq n-i$ and $r(E)-r(J)=s$.
    By the definition of $\iota_s$, we have $\abs{J} \leq \iota_s$.
    Therefore $i_s=n-\iota_s\leq n-\abs{J}\leq n-(n-i)=i$.
\end{proof}
Lemma \ref{lem-degreesfm-onebyone} immediately implies Theorem \ref{thm-degreesfm-onebyone}.
Furthermore, indices $i_0,i_1,\dots, i_{r(E)}$ do not depend on $m$.
In other words, for each degree $s\in\{0,1,\dots,r(E)\}$, there exists an index corresponding to that degree, independent of $m$.
Note that Theorem~\ref{thm-degreesfm-onebyone} does not show that the degree of the polynomial is monotonic with respect to $i$. Indeed, there are examples where $\deg f^m_{i-1}> \deg f^m_{i}$ for some $i\in [n]$.
For details, see Proposition~\ref{prop-vanished-weight-Nk-Zk} and Example~\ref{ex-Z5}.

\subsection{Periods of minimum weights}\label{sec-min-preiod-weight}
In what follows, we use the fact that $A_{G,i}(q)$ is a quasi-polynomial to discuss the minimum weight of the code.

\begin{defi}
    For $m\in[\rho_0]$, we define $d_m^{\prime}$ as the smallest positive integer $i>0$ such that $f_{i}^{m}(t)$ is a nonzero polynomial, that is,
    \begin{align*}
        d_m^{\prime} \coloneq \min\{i\in\mathbb{Z}_{>0}\mid f_{i}^{m}(t)\neq 0\ \text{as a polynomial}\}.
    \end{align*}
\end{defi}
Theorem \ref{thm-degreesfm-onebyone} immediately implies the following.
\begin{cor}\label{cor-deffd'm}
    For any $m\in[\rho_0]$, we have $\deg f^m_{d^{\prime}_m}(t)\leq 1$.
\end{cor}
\begin{proof}
    If $\deg f^m_{d^{\prime}_m}(t)\geq 2$, then $i_1<i_2\leq d^{\prime}_m$ by Theorem \ref{thm-degreesfm-onebyone}.
    Since $f^m_{i_1}(t)$ is a nonzero polynomial, this is a contradiction.
\end{proof}
Note that there are cases where the degree of $f^m_{d^{\prime}_m}(t)$ is $0$
(see Example \ref{example-B2-20/02} \textup{(1)} and Example \ref{ex-kerdock2}).

Furthermore, $f^m_{d^{\prime}_m}(t)$ can be represented by the following closed formula.
\begin{prop}\label{prop-f_d'm}
    For $m\in[\rho_0]$, we have
    \begin{align}\label{eq-f_d'm}
        f_{d_m^{\prime}}^{m}(t) = \sum_{\substack{J \subseteq E \\ \abs{J} = n - d_m^{\prime}}} \left( \frac{\prod_{\ell = 1}^{r(J)} \gcd(m,e_{\ell, J})}{\prod_{\ell = 1}^{r(E)} \gcd(m,e_{\ell, E})}   t^{r(E) - r(J)} - 1 \right).
    \end{align}
\end{prop}
\begin{proof}
    First, we recall that $H_E(q)\subseteq H_K(q)$ for any $K\subseteq E$ and $q\in\mathbb{Z}_{>0}$.
    
    Let $J$ be a subset of $E$ with $\abs{J}=n-d_m^{\prime}$.
    We show $H_E(q) = H_K(q)$ for any $J\subsetneq K\subseteq E$ by contradiction.
    We assume that there exist an integer $q$ with $\gcd(q,\rho_0)=m$ and $q>m$, and a set $K$ with $J\subsetneq K \subseteq E$ such that $H_E(q)\subsetneq H_K(q)$.
    Then, we can take $u\in H_K(q)\setminus H_E(q)$.
    Since $\emptyset\subsetneq \operatorname{supp}(uG)\subseteq E\setminus K$, we have
    \begin{align*}
        0 < \operatorname{wt}(uG) = \abs{\operatorname{supp}(uG)} \leq \abs{E\setminus K} = \abs{E} - \abs{K} < \abs{E} - \abs{J} = d_m^{\prime}.
    \end{align*}
    This means that there exists a codeword $c\in C_G(q)$ such that $0 < \operatorname{wt}(c) < d_m^{\prime}$.
    However, this contradicts the minimality of $d_m^{\prime}$.
    From the above discussion, for any integer $q$ with $\gcd(q,\rho_0)=m$ and $q>m$ and for any set $K$ with $J\subsetneq K \subseteq E$, we have $H_E(q) = H_K(q)$.

    Here, by Proposition \ref{prop-quasipoly-AGi}, for any $q\in\mathbb{Z}_{>0}$ with $\gcd(q,\rho_0)=m$ and $q>m$,
    \begin{align}\label{eq-f_d'mq=sum}
        f_{d_m^{\prime}}^{m}(q) &= A_{G,d_m^{\prime}}(q) = \sum_{\substack{J \subseteq E \\ \abs{J} = n - d_m^{\prime}}}\sum_{K:J \subseteq K \subseteq E} (-1)^{\abs{K} - \abs{J}} \frac{\abs{H_K(q)}}{\abs{H_E(q)}}\notag\\
        &= \sum_{\substack{J \subseteq E \\ \abs{J} = n - d_m^{\prime}}}\left(\frac{\abs{H_J(q)}}{\abs{H_E(q)}} + \sum_{K: J \subsetneq K \subseteq E} (-1)^{\abs{K} - \abs{J}} \right)\notag\\
        &= \sum_{\substack{J \subseteq E \\ \abs{J} = n - d_m^{\prime}}}\left(\frac{\abs{H_J(q)}}{\abs{H_E(q)}} - 1 + \sum_{K:J \subseteq K \subseteq E} (-1)^{\abs{K} - \abs{J}}\right)\notag\\
        &= \sum_{\substack{J \subseteq E \\ \abs{J} = n - d_m^{\prime}}}\left(\frac{\prod_{\ell = 1}^{r(J)} \gcd(m,e_{\ell, J})}{\prod_{\ell = 1}^{r(E)} \gcd(m,e_{\ell, E})}   q^{r(E) - r(J)} -1 \right),
    \end{align}
    where, in the above equalities, we use the fact that
    \begin{align*}
        \sum_{K:J \subseteq K \subseteq E} (-1)^{\abs{K} - \abs{J}} = \sum_{i=0}^{\abs{E\setminus J}}\binom{\abs{E\setminus J}}{i}(-1)^i=(1-1)^{\abs{E\setminus J}}=0.
    \end{align*}
    Since Equation \eqref{eq-f_d'mq=sum} holds for infinitely many $q$, we see that
    \begin{align*}
        f_{d_m^{\prime}}^{m}(t) = \sum_{\substack{J \subseteq E \\ \abs{J} = n - d_m^{\prime}}}\left(\frac{\prod_{\ell = 1}^{r(J)} \gcd(m,e_{\ell, J})}{\prod_{\ell = 1}^{r(E)} \gcd(m,e_{\ell, E})}   t^{r(E) - r(J)} -1 \right)
    \end{align*}
    as a polynomial.
\end{proof}
Since the coefficient of $t^{r(E)-r(J)}$ is positive for any term of the right-hand side of Equation \eqref{eq-f_d'm}, Corollary \ref{cor-deffd'm} and Proposition \ref{prop-f_d'm} imply the following corollary.
\begin{cor}\label{cor-rE-rJ<=1}
    For any $J\subseteq E$ with $\abs{J}=n-d^{\prime}_m$, we have $0\leq r(E)-r(J)\leq 1$.
\end{cor}

\begin{lem}\label{lem-ejJ/ejJ'-in-Z}
Let $J\subseteq J^{\prime}\subseteq E$.
For any $0\leq j\leq r(J)$, the product of the elementary divisors $e_{1,J},\dots e_{j,J}$ is divisible by the product of the elementary divisors $e_{1,J^{\prime}},\dots e_{j,J^{\prime}}$.
In other words,
\begin{align}\label{eq-ejJ/ejJ'-in-Z}
    \frac{\prod_{\ell=1}^{j}e_{\ell,J}}{\prod_{\ell=1}^{j}e_{\ell,J^{\prime}}}=\frac{e_{1,J}\cdots e_{j,J}}{e_{1,J^{\prime}}\cdots e_{j,J^{\prime}}}\in\mathbb{Z}_{>0}.
\end{align}
\end{lem}
\begin{proof}
    If $j=0$, then $\prod_{\ell=1}^{j}e_{\ell,J}=\prod_{\ell=1}^{j}e_{\ell,J^{\prime}}=1$ and Equation \eqref{eq-ejJ/ejJ'-in-Z} holds.
    Let $j$ be an integer with $1\leq j\leq r(J)$.
    From the fundamental theorem for finitely generated abelian groups, $e_{1,J}\cdots e_{j,J}$ is equal to the gcd of all $j\times j$ minors of $G_J$ (see~\cite[Proposition 8.1]{Miller-Reiner09} or~\cite[Theorem 2.4]{Stanley2016}).
    Since $J\subseteq J^{\prime}$, the gcd of all $j\times j$ minors of $G_{J^{\prime}}$ is a divisor of the gcd of all $j\times j$ minors of $G_J$.
    Therefore, we have Equation \eqref{eq-ejJ/ejJ'-in-Z}.
\end{proof}

\begin{lem}\label{lem-gcd-qejJ/qejJ'}
    Let $J\subseteq J^{\prime}\subseteq E$ and let $q\in\mathbb{Z}_{>0}$.
    Then we have
    \begin{align}\label{eq-gcd-qejJ/qejJ'}
    \frac{\prod_{\ell=1}^{r(J)}\gcd(q,e_{\ell,J})}{\prod_{\ell=1}^{r(J)}\gcd(q,e_{\ell,J^{\prime}})}=\frac{\gcd(q,e_{1,J})\cdots \gcd(q,e_{r(J),J})}{\gcd(q,e_{1,J^{\prime}})\cdots \gcd(q,e_{r(J),J^{\prime}})}\in\mathbb{Z}_{>0}.
\end{align}
\end{lem}
\begin{proof}
    The case where $J=\emptyset$ is clear, and so we assume that $J\neq \emptyset$.
    Let $p$ be a prime number.
    For any $i$ with $1\leq i\leq r(J)$, let $\alpha_i$ and $\beta_i$ denote the exponent of the prime $p$ in the prime factorization of $e_{i,J}$ and $e_{i,J^{\prime}}$, respectively.
    Let $\gamma$ denote the exponent of the prime $p$ in the prime factorization of $q$.
    Since $e_{1,J}|e_{2,J}| \cdots |e_{r(J),J}$ and $e_{1,J^{\prime}}|e_{2,J^{\prime}}| \cdots |e_{r(J),J^{\prime}}$, we have $0\leq \alpha_1\leq \alpha_2\leq\cdots\leq\alpha_{r(J)}$ and $0\leq \beta_1\leq \beta_2\leq\cdots\leq\beta_{r(J)}$.
    To prove Equation \eqref{eq-gcd-qejJ/qejJ'}, it suffices to prove that
    \begin{align}\label{eq-min-alpha-beta-gamma}
        \min \{ \beta_{1}, \gamma \} + \cdots + \min \{ \beta_{r(J)}, \gamma \} \leq \min \{ \alpha_{1}, \gamma \} + \cdots + \min \{ \alpha_{r(J)}, \gamma \}.
    \end{align}
    Define $i_{\alpha} \coloneq \min \{ i \in [r(J)+1]\mid \gamma < \alpha_{i}  \}$ and $i_{\beta} \coloneq \min \{ i \in [r(J)+1]\mid \gamma < \beta_{i} \}$ where we assume $\alpha_{r(J)+1}=\beta_{r(J)+1}=\gamma+1$.

    Let us consider the case where $i_{\beta}\leq i_{\alpha}$.
    By definition, $\gamma < \beta_{i}$ for any $i \geq i_{\beta}$.
    In addition, by Lemma \ref{lem-ejJ/ejJ'-in-Z}, $\beta_1 + \cdots + \beta_j \leq \alpha_1 + \cdots +\alpha_j$ for any $0 \leq j\leq r(J)$.
    Therefore, we have
    \begin{align*}
        \min \{ \beta_{1}, \gamma \} + \cdots + \min \{ \beta_{r(J)}, \gamma \} &= \beta_1 + \cdots + \beta_{i_{\beta}-1} + (r(J)-i_{\beta}+1)\gamma\\
        &\leq \beta_1 + \cdots + \beta_{i_{\beta}-1} + \beta_{i_{\beta}} + \cdots + \beta_{i_{\alpha}-1} + (r(J)-i_{\alpha}+1)\gamma\\
        &\leq \alpha_1 + \cdots + \alpha_{i_{\alpha}-1} + (r(J)-i_{\alpha}+1)\gamma\\
        &= \min \{ \alpha_{1}, \gamma \} + \cdots + \min \{ \alpha_{r(J)}, \gamma \}.
    \end{align*}

    Next, consider the case where $i_{\alpha} < i_{\beta}$.
    By definition, $\beta_{i} \leq \gamma$ for any $i < i_{\beta}$.
    Therefore, by Lemma \ref{lem-ejJ/ejJ'-in-Z},
    \begin{align*}
        \min \{ \beta_{1}, \gamma \} + \cdots + \min \{ \beta_{r(J)}, \gamma \} &= \beta_1 + \cdots + \beta_{i_{\alpha}-1} + \beta_{i_{\alpha}} + \cdots + \beta_{i_{\beta}-1} + (r(J)-i_{\beta}+1)\gamma\\
        &\leq \beta_1 + \cdots + \beta_{i_{\alpha}-1} + (r(J)-i_{\alpha}+1)\gamma\\
        &\leq \alpha_1 + \cdots + \alpha_{i_{\alpha}-1} + (r(J)-i_{\alpha}+1)\gamma\\
        &= \min \{ \alpha_{1}, \gamma \} + \cdots + \min \{ \alpha_{r(J)}, \gamma \}.
    \end{align*}
    From the above discussion, we obtain Equation \eqref{eq-min-alpha-beta-gamma}.
\end{proof}

\begin{thm}\label{thm-mindist-ifandonlyif}
    Let $m\in [\rho_0]$ be a divisor of $\rho_0$.
    The minimum weight $d_q$ is the constant $d^\prime_m$ for any $q\in \mathbb{Z}_{>0}$ with $q>m$ and $\gcd(q,\rho_0)=m$.
    Furthermore, $d_m=d^\prime_m$ holds if and only if there exists a subset $J$ of $E$ with $\abs{J}=n-d^{\prime}_m$ such that either \textup{(i)} or \textup{(ii)} below holds:
        \begin{itemize}
            \item[\textup{(i)}] $\displaystyle \frac{\prod_{\ell=1}^{r(J)}\gcd(m,e_{\ell,J})}{\prod_{\ell=1}^{r(J)}\gcd(m,e_{\ell,E})}\neq 1$;
            \item[\textup{(ii)}] $r(E)=r(J)+1$ and $m\nmid e_{r(J)+1,E}$.
        \end{itemize}
\end{thm}
\begin{proof}
    By the definition of $d^{\prime}_{m}$, for any $q\in\mathbb{Z}_{>0}$ with $\gcd(q,\rho_0)=m$, if $A_{G,d^{\prime}_m}(q)=f_{d^{\prime}_m}^m(q)\neq 0$, then we have $d_q=d^{\prime}_m$.
    In other words, if there exist no solutions $q$ of the equation $f_{d^{\prime}_m}^m(t)=0$ such that $q\in\mathbb{Z}_{>0}$ and $\gcd(q,\rho_0)=m$, then the minimum weight $d_q$ coincides with $d^{\prime}_m$.
    
    Let $J$ be a subset of $E$ with $\abs{J}=n - d^{\prime}_m$.
    Define
    \begin{align*}
        g_J(t) \coloneq \frac{\prod_{\ell = 1}^{r(J)} \gcd(m,e_{\ell, J})}{\prod_{\ell = 1}^{r(J)} \gcd(m,e_{\ell, E})} \cdot \frac{t^{r(E) - r(J)}}{\prod_{\ell = r(J)+ 1}^{r(E)} \gcd(m,e_{\ell, E})}.
    \end{align*}
    Then $\displaystyle f_{d_m^{\prime}}^{m}(t)=\sum_{\substack{J \subseteq E \\ \abs{J} = n - d_m^{\prime}}} \left(g_J(t) -1 \right)$ by Proposition \ref{prop-f_d'm}.
    
    Since $0\leq r(E) - r(J)\leq 1$ by Corollary \ref{cor-rE-rJ<=1}, we have $0\leq \deg g_J(t)\leq 1$.
    Let $q\in \mathbb{Z}_{>0}$ be an integer with $\gcd(q,\rho_0)=m$.
    Then
    \begin{align*}
        g_J(q)-1=
        \begin{cases}
            \displaystyle \frac{\prod_{\ell = 1}^{r(J)} \gcd(m,e_{\ell, J})}{\prod_{\ell = 1}^{r(J)} \gcd(m,e_{\ell, E})}-1 &\quad \text{if}\ \ r(E) = r(J),\\[1.5em]
            \displaystyle \frac{\prod_{\ell = 1}^{r(J)} \gcd(m,e_{\ell, J})}{\prod_{\ell = 1}^{r(J)} \gcd(m,e_{\ell, E})}\cdot\frac{q}{\gcd(m,e_{r(J)+1, E})}-1&\quad \text{if}\ \ r(E) = r(J)+1.
        \end{cases}
    \end{align*}
    By Lemma \ref{lem-gcd-qejJ/qejJ'}, we have $g_J(q)-1\geq 0$.
    If $r(E) = r(J)$ for any $J\subseteq E$ with $\abs{J}=n - d^{\prime}_m$, then $\displaystyle f_{d_m^{\prime}}^{m}(q)=f_{d_m^{\prime}}^{m}(t)$ is a nonzero constant by the definition of $d_m^{\prime}$.
    Next, we assume that there exists $J\subseteq E$ with $\abs{J}=n - d^{\prime}_m$ such that $r(E) = r(J) + 1$.
    If $q>m$, then $q>\gcd(m,e_{r(J)+1, E})$.
    Thus, we have $\displaystyle g_J(q)-1 > 0$ so that $\displaystyle f_{d_m^{\prime}}^{m}(q) > 0$.
    Therefore, we have $d_q=d^\prime_m$ for any $q\in \mathbb{Z}_{>0}$ with $q>m$ and $\gcd(q,\rho_0)=m$.

    We now prove the remaining part of the claim.
    Assume that for any $J\subseteq E$ with $\abs{J}=n-d^{\prime}_m$, both \textup{(i)} and \textup{(ii)} do not hold.
    If $r(E)=r(J)$, then $\displaystyle g_J(m)-1=0$ by negation of \textup{(i)}.
    If $r(E)=r(J)+1$, then $\displaystyle g_J(m)-1=0$ since $m|e_{r(J)+1,E}$ means $m=\gcd(m,e_{r(J)+1,E})$.
    This implies $f_{d^{\prime}_m}^m(m)= 0$.
    Therefore, we have $d_m>d^{\prime}_m$.

    Conversely, assume that there exists a subset $J$ of $E$ with $\abs{J}=n-d^{\prime}_m$ such that either \textup{(i)} or \textup{(ii)} holds.
    Note that $m\nmid e_{r(J)+1,E}$ means $m>\gcd(m,e_{r(J)+1,E})$.
    Then we have $g_J(m)-1>0$ in both cases where $r(E)=r(J)$ and $r(E)=r(J)+1$.
    This implies $f_{d^{\prime}_m}^m(m)\neq 0$.
    Therefore, we have $d_m=d^{\prime}_m$.
\end{proof}
Note that the elementary divisors denoted by $e_1, \dots, e_r$ in the notation of Theorem~\ref{thm-sufficient-necessary-mindist} are written as $e_{1,E}, \dots, e_{r(E),E}$ in the notation of this section.
\begin{proof}[Proof of Theorem~\ref{thm-sufficient-necessary-mindist}]
    \textup{(1)} Applying the first half of Theorem~\ref{thm-mindist-ifandonlyif} with $m=1$, $d_q$ is a constant for any $q\in\mathbb{Z}_{>0}$ with $q>1$ and $\gcd(q,\rho_0)=1$.
    The integer $m_0$ is the  smallest integer $q$ satisfying $q > 1$ and $\gcd(q,\rho_0)=1$.
    Therefore, for any $q\in\mathbb{Z}_{>0}$ with $q>1$ and $\gcd(q,\rho_0)=1$, we have $d_q = d_m^{\prime} = d_{m_0}$.
    
    \textup{(2)} Let $m\geq 2$ and $\gcd(m,e_{r(E),E})=1$.
    Assume that for any subset $J$ of $E$ with $\abs{J}=n-d^{\prime}_m$, both \textup{(i)} and \textup{(ii)} of Theorem~\ref{thm-mindist-ifandonlyif} do not hold.
    In the case where there exists $J$ such that $r(E)=r(J)+1$, we have $1=\gcd(m,e_{r(E),E})=\gcd(m,e_{r(J)+1,E})=m\geq 2$ by negation of \textup{(ii)}.
    This is a contradiction.
    Therefore, we have $r(E)=r(J)$ for any $J\subseteq E$ with $\abs{J}=n-d^{\prime}_m$.
    By negation of \textup{(i)},
    \begin{align*}
        f_{d^{\prime}_m}^m(t)=\sum_{\substack{J \subseteq E \\ \abs{J} = n - d_m^{\prime}}}\left(\frac{\prod_{\ell = 1}^{r(J)} \gcd(m,e_{\ell, J})}{\prod_{\ell = 1}^{r(J)} \gcd(m,e_{\ell, E})}-1\right)=0.
    \end{align*}
    This also contradicts the definition of $d^{\prime}_m$.
    Therefore, there exists $J\subseteq E$ with $\abs{J}=n-d^{\prime}_m$, such that either \textup{(i)} or \textup{(ii)} of Theorem~\ref{thm-mindist-ifandonlyif} holds.
    By Theorem~\ref{thm-mindist-ifandonlyif}, we have $d_q=d_m$ for any $q\in\mathbb{Z}_{>0}$ with $\gcd(q,\rho_0)=m$.
    
    \textup{(3)} Assume that $m|e_{1,E}$.
    Since $e_{1,E}|\cdots|e_{r(E),E}$, we have $m|e_{\ell,E}$ for any $\ell\in[r(E)]$.
    Let $J$ be a subset of $E$ with $\abs{J}=n-d^{\prime}_m$.
    If $r(E)=r(J)+1$, then $m|e_{r(J)+1,E}$.
    This is negation of \textup{(ii)} of Theorem~\ref{thm-mindist-ifandonlyif}.
    Moreover, since $e_{1,E}|e_{1,J}$ by Lemma \ref{lem-ejJ/ejJ'-in-Z}, we have $m|e_{1,E}|e_{1,J}|e_{2,J}|\cdots|e_{r(J),J}$.
    Thus, $\gcd(m,e_{\ell,J})=m$ for any $\ell\in[r(J)]$.
    This implies $\displaystyle\frac{\prod_{\ell=1}^{r(J)}\gcd(m,e_{\ell,J})}{\prod_{\ell=1}^{r(J)}\gcd(m,e_{\ell,E})}=\frac{m^{r(J)}}{m^{r(J)}}=1$, which is negation of \textup{(i)} of Theorem~\ref{thm-mindist-ifandonlyif}.
    Therefore, if $r(E)=r(J)+1$, then both \textup{(i)} and \textup{(ii)} do not hold.

    If $r(E)=r(J)$, then \textup{(ii)} does not hold, since $r(E)\neq r(J)+1$.
    For the same reason as the above discussion, \textup{(i)} does not hold.
    Therefore, for any subset $J$ of $E$ with $\abs{J}=n-d^{\prime}_m$, both \textup{(i)} and \textup{(ii)} of Theorem~\ref{thm-mindist-ifandonlyif} do not hold.
    By Theorem~\ref{thm-mindist-ifandonlyif}, we have $d_m \neq d^{\prime}_{m}$.
    Taking $q = m + \rho_{0}$ , we have $q > m$ and $\gcd(q, \rho_{0}) = m$. By Theorem~\ref{thm-mindist-ifandonlyif}, $d_{q} = d_{m}^{\prime} \neq d_{m}$. This proves the contrapositive.
\end{proof}
Theorems~\ref{thm-sufficient-necessary-mindist} and~\ref{thm-mindist-ifandonlyif} immediately imply that the minimum weights over all residue rings are determined by computations over only finitely many residue rings.
The following corollary makes this observation explicit.
\begin{cor}\label{cor-finite-determination}
Let $G\in\Mat_{k\times n}(\mathbb{Z})$, and let $\rho_0$ be the lcm period of
$G$.
Let $e_1\mid e_2\mid\cdots\mid e_r$ be the elementary divisors of $G$, and define
\begin{equation*}
    m_0\coloneqq \min\{q\in\mathbb{Z}_{>1}\mid \gcd(q,\rho_0)=1\},
    \quad \mbox{and} \quad
    \mathcal{D}_G \coloneqq \{m\in\mathbb{Z}_{\geq 2} \mid \gcd(m,e_r)\neq 1,\ m\mid\rho_0\}.
\end{equation*}
Then the minimum weights $d_q$ for all integers $q>1$ are completely
determined by the finite set
\begin{equation*}
    \{d_{m_0}\} \cup \{d_m\mid m\ge2,\ m\mid\rho_0\} \cup \{d_{m+\rho_0}\mid m\in\mathcal D_G\}.
\end{equation*}
Consequently, at most $\tau(\rho_0)+\abs{\mathcal{D}_G}$ minimum-weight computations are sufficient, where $\tau(\rho_0)$ denotes the number of positive divisors of $\rho_0$.

In particular, if $e_r=1$, then $\mathcal{D}_{G} = \emptyset$.
Hence, all minimum weights are determined by $\{d_{m_0}\} \cup \{d_m\mid m\ge2,\ m\mid\rho_0\}$ and it suffices to compute the minimum weights at most $\tau(\rho_0)$ times.
\end{cor}
From a computational point of view, Corollary~\ref{cor-finite-determination} shows that the minimum weights of the family $\{C_G(q)\mid q\in \mathbb{Z}_{\ge 1}\}$ can be determined by finitely many minimum-weight computations.
We summarize the resulting procedure in Algorithm~\ref{alg:minweight}.
\begin{algorithm}[htbp]
\caption{Computing the minimum weights $\{d_q\}_{q\geq 1}$}
\label{alg:minweight}
\begin{algorithmic}[1]
\Require
An integral matrix $G\in\Mat_{k\times n}(\mathbb Z)$.

\Ensure
The minimum weights $d_q$ for all $q>1$.

\State Compute the lcm period $\rho_0$ of $G$ together with the elementary divisors $e_1|\cdots|e_r$ of $G$.
\State Compute
\begin{equation*}
    m_0=\min\{q>1\mid\gcd(q,\rho_0)=1\}.
\end{equation*}
\State Compute $d_{m_0}$.

\ForAll{divisors $m\mid\rho_0$ with $m\ge2$}
    \State Compute $d_m$.
    \If{$\gcd(m,e_r)\neq 1$}
        \State Compute $d_{m+\rho_0}$.
    \EndIf
\EndFor

\State Recover $d_q$ for every $q\geq 1$ using Corollary~\ref{cor-finite-determination}.

\end{algorithmic}
\end{algorithm}
\begin{rem}
The main computational bottleneck in Algorithm~\ref{alg:minweight} is the computation of the lcm period $\rho_0$.
Indeed, computing $\rho_0$ directly from its definition generally requires examining all nonempty column subsets of $G$, and hence may require exponential time in the number of columns.
Thus, Algorithm~\ref{alg:minweight} should be understood as a finite determination procedure once $\rho_0$ has been obtained, rather than as an efficient algorithm for computing $\rho_0$ itself.
\end{rem}
\begin{rem}
For a fixed prime $p$ , the prime powers $q=p^s$ form a family over the finite chain rings $\mathbb Z_{p^s}$.
Let $v_{p}(\rho_{0})$ be the exponent of $p$ in the prime factorization of $\rho_0$.
Then $\gcd(p^s,\rho_0)=p^{\min\{s,v_{p}(\rho_{0})\}}$ holds, so the gcd class is fixed for all positive $s$ with $s\ge v_{p}(\rho_{0})$, and in particular for all sufficiently large $s$.
Consequently, the weight enumerators along this prime-power tower are eventually described by a single constituent polynomial, while the minimum weights stabilize for $s > v_{p}(\rho_{0})$.
\end{rem}

\section{The Greene-type identity}\label{sec-Greene}
In this section, we provide an analog of Greene's theorem (Theorem~\ref{thm-Greene}), which describes the relation between the weight enumerators of linear codes and the Tutte polynomials of the matroids associated with those codes.
The Tutte polynomial is one of the most important invariants of matroids and generalizes the chromatic polynomial and the flow polynomial of graphs, among other invariants.
First, we review the classical results.

\begin{defi}
    Let $k$ and $n$ be positive integers with $k\leq n$, let $\mathbb{F}$ be a field, and let $G\in\operatorname{Mat}_{k\times n}(\mathbb{F})$.
    The \emph{vector matroid} of $G$ is the ordered pair $M[G]\coloneq (E,r)$, where $E=[n]$ is the set of column indices of $G$, and $r$ is an integer-valued function on the power set $2^{E}$ defined as $r(K)=\rank_{\mathbb{F}} G_K$ for $K\subseteq E$, where $G_K$ is the submatrix of $G$ consisting of columns indexed by $K$.
    The \emph{Tutte polynomial} of $M[G]$ is the bivariate polynomial
    \begin{equation*}
        \mathcal{T}_{M[G]}(u,v)=\sum_{K\subseteq E}(u-1)^{r(E)-r(K)}(v-1)^{\left\lvert K\right\rvert-r(K)}\in \mathbb{Z}[u,v].
    \end{equation*}
\end{defi}
See \cite{oxley2006matroid} for the general theory of matroids and \cite{BO1992,JP2013} for Tutte polynomials.

\begin{thm}[{\cite{Greene1976}, cf.~\cite[Theorem~5.9]{JP2013}}]\label{thm-Greene}
Let $\mathbb{F}_q$ be the finite field with $q$ elements and
let $C$ be a linear code over $\mathbb{F}_q$ with generator matrix $G\in\operatorname{Mat}_{k\times n}(\mathbb{F}_q)$; that is, $C$ is the linear subspace spanned by the row vectors of $G$ in $\mathbb{F}_q^n$. 
Let $W_{C}(x,y)\coloneq \sum_{c\in C}x^{n-\operatorname{wt}(c)}y^{\operatorname{wt}(c)}$ be the weight enumerator of $C$, where $\operatorname{wt}(c)=\abs{\{i\in [n]\mid c_i\neq 0\}}$.
Then the following holds:
\begin{equation*}
    W_{C}(x,y)=y^{n-r(E)}(x-y)^{r(E)}\mathcal{T}_{M[G]}\left(\frac{x+(q-1)y}{x-y},\frac{x}{y}\right).
\end{equation*}
\end{thm}

The Tutte quasi-polynomial was introduced by Br\"and\'en and Moci~\cite{BM2014} for a list of elements in a finitely generated abelian group, and generalized by Fink and Moci~\cite{FM2016} to a matroid over $\mathbb{Z}$.
However, here we restrict the definition to a list of elements in $\mathbb{Z}^k$, which suffices to discuss the connection with the weight enumerator of a code.
\begin{defi}[{\cite[Section 7]{BM2014}, see also~\cite[Section 7.1]{FM2016}}]
    Let $G=(g^1,\dots, g^n)\in\operatorname{Mat}_{k\times n}(\mathbb{Z})$ and let $\Gamma_K$ be the torsion subgroup $\mathrm{tor}\left(\mathbb{Z}^k/\langle g^j\mid  j\in K\rangle\right)$ of the quotient $\mathbb{Z}^k/\langle g^j\mid  j\in K\rangle$.
    The \emph{Tutte quasi-polynomial} of $G$ is defined as follows:
    \begin{equation*}
        \mathcal{Q}_G(u,v)\coloneq \sum_{K\subseteq E}\frac{\abs{\Gamma_K}}{\abs{(u-1)(v-1)\Gamma_K}}(u-1)^{r(E)-r(K)}(v-1)^{\left\lvert K\right\rvert-r(K)}
    \end{equation*}
    where $E=[n]$.
\end{defi}
Tutte quasi-polynomials are also discussed by Delucchi and Moci~\cite{DM2018} from the perspective of the chromatic and flow quasi-polynomials of CW-complexes.
We now give a quasi-polynomial version of Greene's theorem.

\begin{thm}\label{thm-GreeneQuasi}
Let $G\in \operatorname{Mat}_{k\times n}(\mathbb{Z})$ and let $\varphi_{G,q} \colon \mathbb{Z}_q^k\to \mathbb{Z}_q^n$ be defined as $\varphi_{G,q}(u)=uG$.
Then we have
    \begin{equation}\label{eq-GreeneQuasi}
        W_G(x,y;q)=\frac{q^{k-r(E)}y^{n-{r(E)}}(x-y)^{r(E)}}{\abs{ \operatorname{Ker}\varphi_{G,q}}}\mathcal{Q}_G\left(\frac{x+(q-1)y}{x-y},\frac{x}{y}\right)
    \end{equation}
and
    \begin{equation}\label{eq-GreeneQuasi-2}
        \mathcal{Q}_G\left(u,v\right)=\frac{\abs{ \operatorname{Ker}\varphi_{G,(u-1)(v-1)}}}{(u-1)^{k-r(E)}(v-1)^{k}}W_G(v,1;(u-1)(v-1)).
    \end{equation}
\end{thm}
\begin{proof}
If Equation~\eqref{eq-GreeneQuasi} holds, then Equation~\eqref{eq-GreeneQuasi-2} follows from the direct calculation:
\begin{align*}
    W_G(v,1;(u-1)(v-1))&=\frac{(u-1)^{k-r(E)}(v-1)^{k-r(E)}(v-1)^{r(E)}}{\abs{ \operatorname{Ker}\varphi_{G,(u-1)(v-1)}}}\mathcal{Q}_G\left(\frac{v+(u-1)(v-1)-1}{v-1},v\right)\\
    &=\frac{(u-1)^{k-r(E)}(v-1)^{k}}{\abs{ \operatorname{Ker}\varphi_{G,(u-1)(v-1)}}}\mathcal{Q}_G\left(u,v\right).
\end{align*}
Hence, we show Equation~\eqref{eq-GreeneQuasi}.
Let $e_{1,K},\dots,e_{r(K),K}$ denote the elementary divisors of $G_K$.
Suppose $\displaystyle u=\frac{x+(q-1)y}{x-y}$ and $\displaystyle v=\frac{x}{y}$.
By Lemma \ref{lem-num-ele-H} and the definition of $H_E(q)$, we have
\begin{equation}\label{eq-thm-quasi-greene-1}
    \abs{\operatorname{Ker}\varphi_{G,q}}=\abs{H_E(q)}=q^{k-r(E)}{\prod_{\ell = 1}^{r(E)} \gcd(q,e_{\ell, E})}.
\end{equation}
Since $\mathbb{Z}^k/\langle g^j\mid  j\in K\rangle\cong \mathbb{Z}^{k-r(K)}\oplus\bigoplus_{\ell=1}^{r(K)}\mathbb{Z}/e_{\ell,K}\mathbb{Z}$, it follows that $\Gamma_K\cong \bigoplus_{\ell=1}^{r(K)}\mathbb{Z}/e_{\ell,K}\mathbb{Z}$.
Moreover, the following holds
\begin{equation*}
    \left(u-1\right)\left(v-1\right)\Gamma_K=\frac{qy}{x-y}\cdot \frac{x-y}{y}\Gamma_K=q\bigoplus_{\ell=1}^{r(K)}\mathbb{Z}/e_{\ell,K}\mathbb{Z}\cong \bigoplus_{\ell=1}^{r(K)}\mathbb{Z}/\frac{e_{\ell,K}}{\gcd(q,e_{\ell,K})}\mathbb{Z},
\end{equation*}
which yields
\begin{equation}\label{eq-thm-quasi-greene-2}
    \frac{\abs{\Gamma_K}}{\abs{ (u-1)(v-1)\Gamma_K}}=\frac{\prod_{\ell=1}^{r(K)}e_{\ell,K}}{\prod_{\ell=1}^{r(K)}e_{\ell,K}/\gcd(q,e_{\ell,K})}=\prod_{\ell=1}^{r(K)}\gcd(q,e_{\ell,K}).
\end{equation}
Furthermore, a direct calculation shows that the following holds for each $K\subseteq E$:
\begin{multline}\label{eq-thm-quasi-greene-3}
    y^{n-{r(E)}}(x-y)^{r(E)}\left(\frac{qy}{x-y}\right)^{{r(E)}-r(K)}\left(\frac{x-y}{y}\right)^{\abs{K}-r(K)}
    =q^{{r(E)}-r(K)}y^{n-\abs{K}}(x-y)^{\abs{K}}\\
    =q^{{r(E)}-r(K)}y^{n-\abs{K}}\left(\sum_{J'\subseteq K}x^{\abs{K\setminus J'}}(-y)^{\abs{J'}}\right)
    =\sum_{J'\subseteq K}q^{{r(E)}-r(K)}x^{\abs{K\setminus J'}}y^{n-\abs{K\setminus J'}}(-1)^{\abs{J'}}.
\end{multline}
By substituting Equations \eqref{eq-thm-quasi-greene-1}, \eqref{eq-thm-quasi-greene-2}, and \eqref{eq-thm-quasi-greene-3}, the right-hand side of Equation \eqref{eq-GreeneQuasi} is 
\begin{align*}
    &\frac{q^{k-r(E)}y^{n-{r(E)}}(x-y)^{r(E)}}{\abs{ \operatorname{Ker}\varphi_{G,q}}}\mathcal{Q}_G\left(\frac{x+(q-1)y}{x-y},\frac{x}{y}\right)\\
    &=\sum_{K\subseteq E}\frac{\abs{\Gamma_K}}{\abs{(u-1)(v-1)\Gamma_K}}\frac{q^{k-r(E)}y^{n-{r(E)}}(x-y)^{r(E)}}{\abs{\operatorname{Ker}\varphi_{G,q}}}\left(\frac{qy}{x-y}\right)^{{r(E)}-r(K)}\left(\frac{x-y}{y}\right)^{\abs{K}-r(K)}\\
    &=\sum_{(J',K):J'\subseteq K\subseteq E}\frac{\prod_{\ell=1}^{r(K)}\gcd(q,e_{\ell,K})}{{\prod_{\ell = 1}^{r(E)} \gcd(q,e_{\ell, E})}}q^{{r(E)}-r(K)}x^{\abs{K\setminus J'}}y^{n-\abs{K\setminus J'}}(-1)^{\abs{J'}}
\end{align*}
where the sum is taken over pairs $(J',K)$ of subsets of $E$ satisfying $J'\subseteq K$.
After replacing $J'$ by $K\setminus J$, the right-hand side of Equation \eqref{eq-GreeneQuasi} becomes
\begin{align*}
    &\sum_{(J,K):J\subseteq K\subseteq E}\frac{\prod_{\ell=1}^{r(K)}\gcd(q,e_{\ell,K})}{{\prod_{\ell = 1}^{r(E)} \gcd(q,e_{\ell, E})}}q^{{r(E)}-r(K)}x^{\abs{J}}y^{n-\abs{J}}(-1)^{\abs{K\setminus J}}\\
    &=\sum_{i=0}^n x^{n-i}y^{i}\sum_{\substack{J \subseteq E \\ \abs{J} = n - i}}\sum_{K:J \subseteq K \subseteq E} (-1)^{\abs{K} - \abs{J}} \frac{\prod_{\ell = 1}^{r(K)} \gcd(q,e_{\ell, K})}{\prod_{\ell = 1}^{r(E)} \gcd(q,e_{\ell, E})} q^{r(E) - r(K)}
    =W_G(x,y;q)
\end{align*}
where, in the last line, we used Equation~\eqref{eq-AGi-sum} and $W_G(x,y;q)=\sum_{i=0}^n A_{G,i}(q)x^{n-i}y^{i}$.
\end{proof}

\begin{rem}
    Two matroidal structures can be constructed from an integer matrix $G$: an arithmetic matroid $(M[G],\mu_G)$ and a real polymatroid $(E,r_{G,q})$.
    Both are defined as generalizations of the classical vector matroid associated with $G$.
    Informally, an arithmetic matroid is a matroid that adds information about the torsion arising from a $\mathbb{Z}$-module, while a real polymatroid can be viewed as a generalization obtained by allowing coefficients in a finite ring or a finite group.
    
    An arithmetic matroid is defined as a pair consisting of a matroid $M$ on the ground set $E$ and a multiplicity function $\mu \colon 2^E\to \mathbb{Z}$ satisfying certain axioms.
    For details of arithmetic matroids, see \cite{DAM2013}.
    From an integer matrix $G$, the arithmetic matroid $(M[G],\mu_G)$ is defined as the pair consisting of the vector matroid $M[G]$ over $\mathbb{R}$ and the multiplicity function $\mu_G$ given by $\mu_G(J)=\prod_{\ell=1}^{r(J)}e_{\ell,J}$ for $J\subseteq E$, where $e_{1,J},\dots,e_{r(J),J}$ are the elementary divisors of $G_J$.
    As pointed out by Br\"and\'en and Moci~\cite{BM2014}, the Tutte quasi-polynomial $\mathcal{Q}_G$ is not an invariant of the arithmetic matroid $(M[G],\mu_G)$, but rather an invariant of $G$.
    It was later shown by Fink and Moci \cite{FM2016} that the Tutte quasi-polynomial is an invariant of the matroids over $\mathbb{Z}$.
    The arithmetic matroid $(M[G],\mu_G)$ is constructed from the matroid over $\mathbb{Z}$ represented by $G$, and the arithmetic Tutte polynomial of $(M[G],\mu_G)$, which is one of the invariants of arithmetic matroids, is obtained as a constituent of the Tutte quasi-polynomial.
    
    On the other hand, the Tutte quasi-polynomial is determined by real polymatroids arising from a sequence of linear codes.
    A real polymatroid $(E,r_{G,q})$ is constructed as the pair consisting of the ground set $E=[n]$ and a real-valued function $r_{G,q}$ on the power set $2^E$ as follows: for $J\subseteq E$, let $\pi_J \colon \mathbb{Z}_q^n\to \mathbb{Z}_q^J$ be the projection defined by eliminating the coordinates corresponding to $E\setminus J$, and define $r_{G,q}(J)=\log_q \abs{\pi_J(C_G(q))}$ (see \cite[Section 3]{SWX24}).
    As shown in \cite[Theorem~10.3]{SWX24}, the weight enumerator of $C_G(q)$ is an invariant of $(E,r_{G,q})$.
    Also, Theorem~\ref{thm-GreeneQuasi} shows that the weight enumerator $W_G(x,y;q)$ and the Tutte quasi-polynomial $\mathcal{Q}_G(u,v)$ are equivalent invariants.
    Therefore, the Tutte quasi-polynomial can be recovered from the infinite sequence $\bigl((E,r_{G,q})\bigr)_{q\in \mathbb{Z}_{\ge 2}}$ of real polymatroids.
    Moreover, the sequence $\bigl((E,r_{G,q})\bigr)_{q\in \mathbb{Z}_{\ge 2}}$ determines the arithmetic matroid $(M[G],\mu_G)$, although we do not discuss the details here.
\end{rem}

\begin{rem}
    In the classical case of linear codes over finite fields, Greene's theorem is closely related to the MacWilliams identity through matroid duality.
    In the present setting over $\mathbb{Z}_q$, however, the dual code $C_G(q)^\perp=\{v\in \mathbb{Z}_q^n\mid Gv^\top=0\}$ is naturally described by $G$ as a parity-check matrix rather than as a code generated by a fixed integral generator matrix independent of $q$.
    Therefore, we cannot easily derive the MacWilliams identity from Theorem~\ref{thm-GreeneQuasi}.
    Instead, we regard the MacWilliams identity over the finite Frobenius ring $\mathbb{Z}_q$ as an established coding-theoretic result \cite{Wood1999} and use it to relate our quasi-polynomial formula for $W_{C_G(q)}$ to the weight enumerator of $C_G(q)^\perp$.
\end{rem}

\section{Characteristic quasi-polynomial of certain classes of codes}\label{sec-NkZk}

In this section, we consider two families of codes associated with two families of matroids, denoted by $N_k$ and $Z_k$ for $k\ge 2$. 
As in \cite[Example 9.4.18]{oxley2006matroid}, $N_k$ is a binary matroid represented by the $(0, 1)$-matrix
\begin{equation*}
  \widetilde{N}_{k} = \bigl(I_k \mid J_k - I_k\bigr),
\end{equation*}
where $I_k$ is the $k\times k$ identity matrix and $J_k$ is the $k\times k$ all-ones matrix. 
This matroid is used to show that there is no polynomial $f\in\mathbb{Z}[x]$ such that, for every $n\in\mathbb{Z}_{>0}$ and every matroid $M$ on $n$ elements, at most $f(n)$ queries to an independence oracle suffice to decide whether $M$ is binary (for more details, see also \cite[Proposition 9.4.17]{oxley2006matroid}).

The matroid $Z_k$ is the binary matroid represented by the matrix $\widetilde{Z}_{k}$ obtained from $\widetilde{N}_{k}$ by appending the all-ones column vector $1_k$:
\begin{equation*}
  \widetilde{Z}_{k}=(\widetilde{N}_{k} \mid 1_k)=\bigl(I_k \mid J_k - I_k \mid 1_k\bigr).
\end{equation*}
This matroid is characterized as the unique binary $k$-spike with tip $t$ \cite[Proposition 12.2.20]{oxley2006matroid}.

Let ${1}_{k-1},{0}_{k-1}\in\mathbb{Z}^{k-1}$ be the all-ones and all-zeros column vectors, respectively, and set
\begin{equation*}
  Q_{k} \coloneq
  \begin{pmatrix}
    I_{k-1} & -{1}_{k-1} \\
    {0}_{k-1}^{\top} & -1
  \end{pmatrix} \in \GL(k,\mathbb{Z}).
\end{equation*}
Let $P_{N_{k}}\in\GL(2k,\mathbb{Z})$ be the permutation matrix obtained by swapping the $k$th and $2k$th columns of the identity matrix, and set $D_{N_{k}}\coloneq \diag(I_k,-I_k)\in\GL(2k,\mathbb{Z})$.
Then a direct computation yields
\begin{equation} \label{eq-equiv-Nk}
  Q_{k} \widetilde{N}_{k} P_{N_{k}} D_{N_{k}} = \begin{pmatrix}
    I_{k-1} & {1}_{k-1} & I_{k-1} & {1}_{k-1} \\
    {0}_{k-1}^{\top} & 0 & {1}_{k-1}^{\top} & 1
\end{pmatrix}.
\end{equation}
Similarly, let $P_{Z_{k}} \in \GL(2k+1,\mathbb{Z})$ be the permutation matrix obtained by swapping the $k$th and $2k$th columns of the identity matrix, and set $D_{Z_{k}} \coloneq \diag(I_k,-I_{k+1})\in\GL(2k+1,\mathbb{Z})$.
Then
\begin{equation}\label{eq-equiv-Zk}
  Q_{k} \widetilde{Z}_{k} P_{Z_{k}} D_{Z_{k}} = \begin{pmatrix}
  I_{k-1} & {1}_{k-1} & I_{k-1} & {1}_{k-1} & {0}_{k-1} \\
  {0}_{k-1}^{\top} & 0 & {1}_{k-1}^{\top} & 1 & 1
  \end{pmatrix}.
\end{equation}
Since $Q_{k}$ and $P_{N_{k}}D_{N_{k}}$ (resp. $P_{Z_{k}}D_{Z_{k}}$) are unimodular, their reductions modulo $q$ are invertible for every $q \in \mathbb{Z}_{>0}$.
Consequently, for each $q$, the matrices $\widetilde{N}_{k}$ and $Q_{k}\widetilde{N}_{k}P_{N_{k}}D_{N_{k}}$
yield the $\mathbb{Z}_{q}$-linear codes that have the same weight enumerators. This is also the case for $\widetilde{Z}_{k}$ and $Q_{k}\widetilde{Z}_{k}P_{Z_{k}}D_{Z_{k}}$.
In particular, we may use the transformed matrices as generator matrices for $\mathcal{N}_{k}(q) \coloneq C_{\widetilde{N}_{k}}(q)$ and $\mathcal{Z}_{k}(q) \coloneq C_{\widetilde{Z}_{k}}(q)$.

We now consider some weights of codes associated with
$\widetilde{N}_{k}$ and $\widetilde{Z}_{k}$
by viewing them as integer $(0,1)$-matrices in
$\mathbb{Z}^{k\times 2k}$ and $\mathbb{Z}^{k\times (2k+1)}$, respectively.
Although a complete description of the weight enumerators of $\mathcal{N}_{k}(q)$ and $\mathcal{Z}_{k}(q)$
appears to be delicate, one can still derive simple parity obstructions.

\begin{prop} \label{prop-vanished-weight-Nk-Zk}
  Let $k>1$ and $q\in\mathbb{Z}_{>0}$.
  The code $\mathcal{N}_{k}(q)$ has no codewords of odd weight less than $k$.
  The code $\mathcal{Z}_{k}(q)$ has no codewords of odd weight at most $k$.
  Furthermore, for $k \geq 3$, neither $\mathcal{N}_{k}(q)$ nor $\mathcal{Z}_{k}(q)$ has codewords of weight $2$.
\end{prop}

\begin{proof}
  Write a codeword of $\mathcal{Z}_{k}(q)$ in the form
  \begin{equation*}
    {c}=(u_{1}, \dots, u_{k})\widetilde{Z}_{k} = (u_1,\dots,u_k \mid u_{0}-u_1,\dots,u_{0}-u_k \mid u_{0}),
    \qquad
    \text{where}\quad u_{0} \coloneq \sum_{i=1}^k u_i.
  \end{equation*}
  
  First, assume that $u_{0} = 0$. Then for each $i$, the pair $(u_i,\,u_{0}-u_i)=(u_i,\,-u_i)$ is either
  $(0,0)$ or consists of two nonzero entries. Hence, the contribution of the first $2k$ coordinates
  to the weight is even, and since the last coordinate equals $u_{0}=0$, the total weight of
  $c$ is even. This shows, in particular, that in the corresponding subcase for $\mathcal{N}_{k}(q)$
  (where the last coordinate is absent) the weight is also even.
  
   Next, assume that $u_{0} \neq 0$. For each $i$, at least one of $u_i$ and $u_{0}-u_{i}$ is nonzero
  (otherwise $u_{0} = u_{i} + (u_{0} - u_{i}) = 0$), so each of the $k$ pairs contributes at least one nonzero coordinate.
  Together with the last coordinate $u_{0} \neq 0$, we obtain $\wt({c})\ge k+1$.
  Therefore, $\mathcal{Z}_{k}(q)$ has no codewords of odd weight at most $k$.

  For $\mathcal{N}_{k}(q)$, the same argument applies to the first $2k$ coordinates:
  if $u_{0} \neq 0$, then each pair $(u_i,\,u_{0}-u_i)$ contributes at least one nonzero entry,
  so $\wt(c) \geq k$. Hence $\mathcal{N}_{k}(q)$ has no codewords of odd weight smaller than $k$.

  Finally, we suppose $k\ge 3$ and show that there are no codewords of weight $2$.
  Since $G\in \{\widetilde{N}_k,\widetilde{Z}_k\}$ is of the form $G=(I_k| \:\ast\:)$, we have $\wt(u)\le \wt(uG)$ for any $u\in \mathbb{Z}_q^k$.
  Thus, if there is $u\in \mathbb{Z}_q^k$ such that $\wt(uG)=2$, then $\wt(u)\le 2$.
  Without loss of generality, suppose $u_3=\dots=u_k=0$.
  If $u_1\neq 0$ and $u_2=0$, then we have $\wt(u\widetilde{N}_k)=k$ and $\wt(u\widetilde{Z}_k)=k+1$ by direct calculation.
  If both $u_1$ and $u_2$ are nonzero, then we have
  \begin{align*}
      \wt(u\widetilde{N}_k)&=\wt(u_1,u_2,0,\dots,0\mid u_2,u_1,u_1+u_2,\dots,u_1+u_2)\ge 4,\\
      \wt(u\widetilde{Z}_k)&=\wt(u_1,u_2,0,\dots,0\mid u_2,u_1,u_1+u_2,\dots,u_1+u_2\mid u_1+u_2)\ge 4
  \end{align*}
  as required.
\end{proof}

Since $\wt((1,-1,0,\dots,0)\widetilde{N}_k)=\wt((1,-1,0,\dots,0)\widetilde{Z}_k)=4$ for any $k\ge 4$, we have the following corollary.
\begin{cor}
Let $k>1$ and $q>1$.
Let $d^{\widetilde{N}_k}_q$ and $d^{\widetilde{Z}_k}_q$ be the minimum weights of $\mathcal{N}_k(q)$ and $\mathcal{Z}_k(q)$, respectively.
We have
    \begin{equation*}
        d^{\widetilde{N}_k}_q=\begin{cases}
            k&(k< 4),\\
            4&(k\ge 4),
        \end{cases}\qquad\text{and}\qquad
        d^{\widetilde{Z}_k}_q=\begin{cases}
            k+1&(k< 4),\\
            4&(k\ge 4).
        \end{cases}
    \end{equation*}
\end{cor}

We close this section with the computation of
the characteristic quasi-polynomials associated with
$\widetilde{N}_{k}$ and $\widetilde{Z}_{k}$.
For $G\in\{\widetilde{N}_{k},\,\widetilde{Z}_{k}\}$,
the $\mathbb{Z}$-homomorphism $\varphi_{G, q}$ is injective.
Hence, by Corollary~\ref{cor-weight=n}, for every $q\in\mathbb{Z}_{>0}$ we have
\begin{equation*}
  \abs{\mathbb{Z}_{q}^{k} \setminus \bigcup_{j \in E} H_{j}^{G}(q)} = A_{G, n}(q),
\end{equation*}
where we now write $H_{j}^{G}(q)$ (instead of $H_j(q)$ as in the previous sections) to make the dependence on the chosen matrix $G$ explicit.
Here, $E$ denotes the set of column indices of $G$ and $n\coloneq\abs{E}$.

To this end, we use additive characters of $\mathbb{Z}_q$.
In the remainder of this section, we sometimes identify $\mathbb{Z}_q$ with $\{0,1,\dots,q-1\}\subseteq\mathbb{Z}$.
For $q\in\mathbb{Z}_{>0}$, let $\zeta_q \coloneq \exp(2\pi i/q)$ be a primitive $q$-th root of unity.
For each $r\in\mathbb{Z}_q$, define a character $\psi_r \colon \mathbb{Z}_q\to\mathbb{C}^\times$ as
\begin{equation*}
  \psi_r(a)\coloneq \zeta_q^{ra}=\exp\!\left(\frac{2\pi i}{q}\,ra\right)
  \qquad(a\in\mathbb{Z}_q).
\end{equation*}
These characters satisfy the orthogonality relations
\begin{equation*}
  \frac{1}{q}\sum_{a\in\mathbb{Z}_q}\psi_r(a)=
  \begin{cases}
    1 & \text{if } r=0,\\
    0 & \text{if } r\neq 0,
  \end{cases}
  \quad \text{and} \quad
  \frac{1}{q}\sum_{r\in\mathbb{Z}_q}\psi_r(a)=
  \begin{cases}
    1 & \text{if } a=0,\\
    0 & \text{if } a\neq 0.
  \end{cases}
\end{equation*}
Note that, for all $a,a_1,a_2,r,r_1,r_2\in\mathbb{Z}_q$, we have
\begin{equation*}
  \psi_{r_1+r_2}(a) = \psi_{r_1}(a)\psi_{r_2}(a),\qquad
  \psi_r(a_1+a_2) = \psi_r(a_1)\psi_r(a_2),\qquad
  \psi_r(a) = \zeta_q^{ra}=\psi_a(r).
\end{equation*}

For $q\in\mathbb{Z}_{>0}$ and $G\in\{\widetilde{N}_{k},\,\widetilde{Z}_{k}\}$, set
\begin{equation*}
  \chi_{G}^{\mathrm{quasi}}(q) \coloneq \left|\mathbb{Z}_q^{k}\setminus\bigcup_{j\in E(G)}H_{G,j}(q)\right|.
\end{equation*}
We are now ready to state explicit formulas for the characteristic quasi-polynomials
$\chi_{G}^{\mathrm{quasi}}(q)$ for $G \in \{ \widetilde{N}_{k},\, \widetilde{Z}_{k} \}$, using character sums.

\begin{thm}\label{thm-char-quasi-for-Nk-and-Zk}
  For every $k>1$ and $q\in\mathbb{Z}_{>0}$, we have
  \begin{align*}
    &\chi_{\widetilde{Z}_{k}}^{\mathrm{quasi}}(q)
    =
    \frac{(q - 1)(q - 2)^{k} - (-2)^{k}(2q - 1)}{q}+(-1)^{k}\sum_{j = 0}^{k} \binom{k}{j}\gcd(j - 1, q), \\
    &\chi_{\widetilde{N}_{k}}^{\mathrm{quasi}}(q) = \frac{(q -1)^{k} + (-1)^{k}(q - 1)}{q} + \chi_{\widetilde{Z}_{k}}^{\mathrm{quasi}}(q),
  \end{align*}
where $\gcd(-1, q)=1$ and $\gcd(0, q)=q$.
\end{thm}

\begin{proof}
  For notational simplicity, we henceforth replace $\widetilde{N}_{k}$ and $\widetilde{Z}_{k}$ by the unimodularly row-equivalent and signed-column-permuted matrices $Q_k\widetilde{N}_{k}P_{N_k}D_{N_k}$ and $Q_k\widetilde{Z}_{k}P_{Z_k}D_{Z_k}$ as shown in Equation \eqref{eq-equiv-Nk} and \eqref{eq-equiv-Zk}, and denote them again by $\widetilde{N}_{k}$ and $\widetilde{Z}_{k}$.
  Then, for every $q \in \mathbb{Z}_{>0}$, let $\mathcal{N}_{k}(q)$ and $\mathcal{Z}_{k}(q)$ be
  the $\mathbb{Z}_{q}$-code generated by $\widetilde{N}_{k}$ and $\widetilde{Z}_{k}$, respectively.
  We label the columns of $\widetilde{N}_{k}$ from left to right by $x_1,\dots,x_k,\; y_1,\dots,y_k$,
  and write $E= \{x_1,\dots,x_k,\,y_1,\dots,y_k\}$ for the resulting ground set.
  For $\widetilde{Z}_{k}$, we label the additional (last) column by $t$ and set
  $\overline{E} \coloneq E \sqcup \{ t \}$ so that the columns of $\widetilde{Z}_{k}$ are indexed by
  $x_1,\dots,x_k,\,y_1,\dots,y_k,\,t$ in this order.
  
  Consider $\widetilde{Z}_{k}$ first. From the above observation, it suffices to count the number $A_{\widetilde{Z}_{k}, 2k+1}(q)$ of codewords having full weight.
  Note that all codewords ${c} = (c_{x_{1}}, \dots, c_{x_{k}}, c_{y_{1}}, \dots, c_{y_{k}}, c_{t}) \in \mathcal{Z}_{k}(q)$ are of the form
  \begin{equation*}
    {c} = {u}\widetilde{Z}_{k} = (u_{1}, \dots, u_{k-1}, \gamma)
    \bordermatrix{
      & x_{1} \dots x_{k-1} & x_{k} & y_{1} \dots y_{k-1} & y_{k} & t \cr
      & I_{k - 1} & {1}_{k-1} & I_{k-1} & {1}_{k-1} & {0}_{k-1} \cr
      & {0}_{k-1}^{\top} & 0 & {1}_{k-1}^{\top} & 1 & 1
    }
  \end{equation*}
  for some ${u} \coloneq (u_{1}, \dots, u_{k - 1}, \gamma) \in \mathbb{Z}_{q}^{k}$.
  Then we have
  \begin{equation*}
    \left\{
      \begin{aligned}
        c_{x_{i}} &= u_{i} & &\; \text{for all } i \in \{ 1, \dots, k - 1 \}, &\;
        c_{x_{k}} &= \sum_{i=1}^{k-1} u_{i}, \\
        c_{y_{i}} &= u_{i} + \gamma & &\;\text{for all } i \in \{ 1, \dots, k - 1 \}, &\; 
        c_{y_{k}} &= \sum_{i=1}^{k - 1} u_{i} + \gamma, & c_{t} &= \gamma.
      \end{aligned}
    \right.
  \end{equation*}
  Thus $\supp({c})=\overline{E}$ if and only if $\gamma\neq 0$,
  $u_i\notin\{0,-\gamma\}$ for all $i\in\{1,\dots,k-1\}$, and
  $\sum_{i=1}^{k-1}u_i\notin\{0,-\gamma\}$.
 Consequently, for $s \in \mathbb{Z}_{q}$ and $\gamma \in \mathbb{Z}_{q} \setminus \{ 0 \}$, set
  \begin{equation*}
    \mathcal{U}_{q}^{k - 1}(s; \gamma) \coloneq \left\{
      (u_{1}, \dots, u_{k-1}) \in (\mathbb{Z}_{q} \setminus \{ 0, -\gamma \})^{k-1}
      \mathrel{}\middle|\mathrel{}
      \sum_{i=1}^{k - 1} u_{i} = s
    \right\}.
  \end{equation*}
  Then the number of codewords that have full weight is given by
  \begin{equation*}
    A_{\widetilde{Z}_{k}, 2k+1}(q) = \sum_{\gamma \in \mathbb{Z}_{q} \setminus \{ 0 \}}
    \left(
      \abs{\mathbb{Z}_{q} \setminus \{ 0, -\gamma \}}^{k - 1}
      - \abs{\mathcal{U}_{q}^{k-1}(0; \gamma)}
      - \abs{\mathcal{U}_{q}^{k - 1}(-\gamma; \gamma)}
    \right).
  \end{equation*}
  For $s \in \mathbb{Z}_{q}$ and $\gamma \in \mathbb{Z}_{q} \setminus \{ 0 \}$, by applying the orthogonality relation repeatedly, we obtain
  \begin{align*}
    \abs{\mathcal{U}_{q}^{k - 1}(s; \gamma)}
    &= \sum_{u_{1}, \dots, u_{k-1} \in \mathbb{Z}_{q} \setminus \{ 0, -\gamma \}} \frac{1}{q} \sum_{a \in \mathbb{Z}_{q}} \psi_{u_{1} + \dots + u_{k-1} - s}(a)
    = \sum_{u_{1}, \dots, u_{k-1} \in \mathbb{Z}_{q} \setminus \{ 0, -\gamma \}} \frac{1}{q} \sum_{a \in \mathbb{Z}_{q}} \psi_{-s}(a) \prod_{i = 1}^{k - 1} \psi_{u_{i}}(a) \\
    &= \frac{1}{q} \sum_{a \in \mathbb{Z}_{q}} \psi_{-s}(a) \prod_{i = 1}^{k-1} \left( \sum_{u_{i} \in \mathbb{Z}_{q} \setminus \{ 0, -\gamma \}} \psi_{u_{i}}(a) \right)
    = \frac{1}{q} \sum_{a \in \mathbb{Z}_{q}} \psi_{-s}(a) \left( \sum_{r \in \mathbb{Z}_{q} \setminus \{ 0, -\gamma \}} \psi_{r}(a) \right)^{k - 1}.
  \end{align*}
  Similarly, we evaluate the inner sum as follows:
  \begin{equation*}
    \sum_{r \in \mathbb{Z}_{q} \setminus \{ 0, -\gamma \}} \psi_{r}(a)
    = \left( \sum_{r \in \mathbb{Z}_{q}} \psi_{r}(a) \right) - \psi_{0}(a) - \psi_{-\gamma}(a)
    = \begin{cases}
      q - 2 & \text{if } a = 0, \\
      -(1 + \psi_{-\gamma}(a)) & \text{if } a \in \mathbb{Z}_{q} \setminus \{ 0 \}.
    \end{cases}
  \end{equation*}
  Consequently,
  \begin{equation*}
    \abs{\mathcal{U}_{q}^{k-1}(s; \gamma)} = \frac{1}{q} \left[
      (q - 2)^{k - 1} + (-1)^{k-1} \sum_{a \in \mathbb{Z}_{q} \setminus \{ 0 \}} \psi_{-s}(a)(1 + \psi_{-\gamma}(a))^{k - 1}
    \right].
  \end{equation*}
  Combining this with the definition of $\mathcal{U}_q^{k-1}(s;\gamma)$, we get
  \begin{multline*}
    \abs{\mathbb{Z}_{q}\setminus\{0,-\gamma\}}^{k-1}
    -\abs{\mathcal{U}_{q}^{k-1}(0;\gamma)}-\abs{\mathcal{U}_{q}^{k-1}(-\gamma;\gamma)} \\
    = \frac{1}{q}\Bigl[
    q(q-2)^{k-1}-2(q-2)^{k-1} 
    -(-1)^{k-1}\sum_{a\in\mathbb{Z}_q\setminus\{0\}}
    \Bigl(\psi_{0}(a)+\psi_{\gamma}(a)\Bigr)\bigl(1+\psi_{-\gamma}(a)\bigr)^{k-1}
    \Bigr].
  \end{multline*}
  Since $\psi_{0}(a) + \psi_{\gamma}(a) = 1 + \zeta_{q}^{\gamma a} = \zeta_{q}^{\gamma a}(1 + \zeta_{q}^{-\gamma a}) = \psi_{\gamma}(a)(1 + \psi_{-\gamma}(a))$, it follows that
  \begin{align*}
    A_{\widetilde{Z}_{k}, 2k+1}(q) = \frac{1}{q} \left[
      (q - 1)(q - 2)^{k} + (-1)^{k} \sum_{\gamma \in \mathbb{Z}_{q} \setminus \{ 0 \}} \sum_{a \in \mathbb{Z}_{q} \setminus \{ 0 \}} \psi_{\gamma}(a)(1 + \psi_{-\gamma}(a))^{k}
    \right].
  \end{align*}
  Applying the binomial theorem and using
  $\psi_{\gamma}(a)\psi_{-\gamma}(a)^j=\psi_{-\gamma}(a(j-1))$,
  we obtain
  \begin{align*}
    &\sum_{\gamma \in \mathbb{Z}_{q} \setminus \{ 0 \}} \sum_{a \in \mathbb{Z}_{q} \setminus \{ 0 \}} \psi_{\gamma}(a)(1 + \psi_{-\gamma}(a))^{k}
    = \sum_{\gamma \in \mathbb{Z}_{q} \setminus \{ 0 \}} \sum_{a \in \mathbb{Z}_{q} \setminus \{ 0 \}} \psi_{\gamma}(a) \sum_{j = 0}^{k} \binom{k}{j} \psi_{-\gamma}(a)^{j} \\
    ={}& \sum_{\gamma \in \mathbb{Z}_{q} \setminus \{ 0 \}} \sum_{a \in \mathbb{Z}_{q} \setminus \{ 0 \}} \sum_{j=0}^{k} \binom{k}{j} \psi_{-\gamma}(a(j - 1))
    = \sum_{j = 0}^{k} \binom{k}{j} \sum_{\gamma \in \mathbb{Z}_{q} \setminus \{ 0 \}} \sum_{a \in \mathbb{Z}_{q} \setminus \{ 0 \}} \psi_{-\gamma}(a(j - 1)).
  \end{align*}
  By applying the orthogonality relation again, we deduce
  \begin{align*}
    \sum_{a \in \mathbb{Z}_{q} \setminus \{ 0 \}} \psi_{-\gamma}(a(j - 1))
    &= \sum_{a \in \mathbb{Z}_{q} \setminus \{ 0 \}} \psi_{-\gamma(j - 1)}(a)
    = \sum_{a \in \mathbb{Z}_{q}} \psi_{-\gamma(j - 1)}(a) - \psi_{-\gamma(j - 1)}(0) \\
    &= \begin{cases}
      q - 1 & \text{if } -\gamma(j - 1) \equiv 0 \pmod{q}, \\
      -1 & \text{if } -\gamma(j - 1) \not\equiv 0 \pmod{q}.
    \end{cases}
  \end{align*}
  Set $m_{j - 1} \coloneq \gcd(j - 1, q)$, and then let $\overline{j - 1} \coloneq (j - 1)/m_{j-1}$ and $\overline{q} \coloneq q/m_{j-1}$. Note that
  \begin{equation*}
    -\gamma(j - 1) \equiv 0 \pmod{q}
    \iff \gamma \cdot \overline{j - 1} \equiv 0 \pmod{\overline{q}} 
    \iff \gamma \equiv 0 \pmod{\overline{q}}.
  \end{equation*}
  Then, the number of $\gamma \in \mathbb{Z}_{q} \setminus \{ 0 \}$ that satisfy $\gamma \equiv 0 \pmod{\overline{q}}$ is $m_{j - 1} - 1$, that is, $\overline{q}, 2\overline{q}, \dots, (m_{j-1} - 1) \overline{q}$.
  The number of the remaining elements is $q - 1 - (m_{j - 1} - 1) = q - m_{j - 1}$, and they satisfy $\gamma(j - 1) \not\equiv 0 \pmod{q}$. Consequently,
  \begin{equation*}
    \sum_{\gamma \in \mathbb{Z}_{q} \setminus \{ 0 \}} \sum_{a \in \mathbb{Z}_{q} \setminus \{ 0 \}} \psi_{-\gamma}(a(j - 1))
    = (m_{j - 1} - 1)(q - 1) + (q - m_{j - 1}) \cdot (-1) 
    = q\gcd(j - 1, q) - 2q + 1.
  \end{equation*}
  Therefore, we have
  \begin{align*}
    \chi_{\widetilde{Z}_{k}}^{\mathrm{quasi}}(q) &= A_{\widetilde{Z}_{k}, 2k+1}(q)
     = \frac{1}{q}\left[
      (q - 1)(q - 2)^{k}
      + (-1)^{k} \sum_{j = 0}^{k} \binom{k}{j} \bigl(q\gcd(j - 1, q) - 2q + 1\bigr)
     \right]\\
     &=\frac{(q - 1)(q - 2)^{k}}{q}+(-1)^{k}\sum_{j = 0}^{k} \binom{k}{j}\gcd(j - 1, q)-\frac{(-1)^{k}(2q - 1)}{q}\sum_{j = 0}^{k} \binom{k}{j}\\
     &=\frac{(q - 1)(q - 2)^{k}-(-2)^{k}(2q - 1)}{q}+(-1)^{k}\sum_{j = 0}^{k} \binom{k}{j}\gcd(j - 1, q),
  \end{align*}
  as required.

  For $\widetilde{N}_k$, we proceed as in the case of $\widetilde{Z}_k$, with the only difference being that
  the parameter $\gamma$ may now be zero (since there is no coordinate $t$).
  When $\gamma=0$, we have $\supp({c})=E$ if and only if $u_i\neq 0$
  for all $i \in [k-1]$ and $\sum_{i=1}^{k-1}u_i\neq 0$.
  The number of such codewords is given by
  the characteristic polynomial of the uniform matroid $U_{k-1,k}$
  (equivalently, the reduced chromatic polynomial of the cycle graph $C_k$), namely
  $\frac{(q-1)^k+(-1)^k(q-1)}{q}$ (see \cite[Example 7.2.2]{White1987}).
  Therefore,
  \begin{equation*}
    \chi_{\widetilde{N}_{k}}^{\mathrm{quasi}}(q)
    =A_{\widetilde{N}_{k}, 2k}(q) =\frac{(q-1)^k+(-1)^k(q-1)}{q} + A_{\widetilde{Z}_{k}, 2k+1}(q),
  \end{equation*}
  which completes the proof.
\end{proof}

As an immediate consequence of Theorem~\ref{thm-char-quasi-for-Nk-and-Zk}, we can read off the minimum periods of the resulting quasi-polynomials.

\begin{cor}\label{cor-period-for-Nk-and-Zk}
  Let $k>1$. The weight enumerators $W_{\widetilde{N}_{k}}(x, y; q)$ and $W_{\widetilde{Z}_{k}}(x, y; q)$ have the same minimum period
  $\rho_0=\lcm(1,2,\dots,k-1)$.
\end{cor}

\begin{proof}
  It is sufficient to prove that the characteristic quasi-polynomials
  $\chi_{\widetilde{N}_{k}}^{\mathrm{quasi}}(q)$ and 
  $\chi_{\widetilde{Z}_{k}}^{\mathrm{quasi}}(q)$ have
  the minimum period $\lcm(1,2,\dots,k-1)$.
  In the formulas of Theorem~\ref{thm-char-quasi-for-Nk-and-Zk},
  the dependence on the residue classes of $q$ arises only through the terms $\gcd(j-1,q)$.
  (Here the case $j=1$ gives $\gcd(0,q)=q$, which contributes only polynomial terms
  and hence does not affect the period.)
  For each integer $m \ge 1$, the function $q \mapsto \gcd(m,q)$ is periodic with period $m$;
  therefore, $\lcm(1,2,\dots,k-1)$ is a period of both quasi-polynomials.

  Let $L \coloneqq \lcm(1, 2, \dots, k - 1)$,
  and define
  $$
  S_k(q) \coloneqq \sum_{\substack{0\le j\le k\\ j\neq 1}}
  \binom{k}{j}\gcd(j-1,q).
  $$
  Here, since the term corresponding to $j=1$ is $\binom{k}{1}\gcd(0,q)=kq$, which is a polynomial in $q$, it does not affect the minimum period.
  Since $\gcd(j-1,q)\le |j-1|$ unless $j = 1$, equality holds
  simultaneously for all terms if and only if $L\mid q$.
  Hence, $S_k(q)$ attains its maximum value if and only if $L\mid q$.

  Let $\rho$ be a period of $S_k$.
  Since $S_k(L+\rho)=S_k(L)$, the maximality of $S_k(L)$ implies
  $L\mid (L+\rho)$, and therefore $L\mid \rho$.
  Hence, the minimum period of $S_k$ is $L$.
  
  Therefore, the characteristic quasi-polynomials
  $\chi^{\mathrm{quasi}}_{\widetilde{N}_k}(q)$ and
  $\chi^{\mathrm{quasi}}_{\widetilde{Z}_k}(q)$ have minimum period $L$.
  By Corollary~\ref{cor-weight-minimum-period},
  the weight enumerators
  $W_{\widetilde{N}_k}(x,y;q)$ and
  $W_{\widetilde{Z}_k}(x,y;q)$
  also have minimum period $L$.
\end{proof}

\section{Coding-theoretic and computational examples}\label{sec-examples}
In this section, we present several examples illustrating the main results of
this paper from both coding-theoretic and computational viewpoints.
The weight enumerators are computed using SageMath~\cite{SageMath}.

In all the following examples, when $\gcd(q,\rho_0)=1$, every nonzero
coefficient $A_{G,i}(q)$ ($i>0$) is divisible by $q-1$.
This reflects the fact that $C_G(1)=\{0\}$ for every
$G\in\Mat_{k\times n}(\mathbb Z)$.

\subsection{Quasi-polynomial structure of the extended Hamming code}

The first example illustrates how the quasi-polynomial viewpoint reveals information that is invisible over a single residue ring.
We consider the $[8,4]$ extended Hamming code over $\mathbb Z_2$, whose weight distribution is completely understood in classical coding theory.
\begin{example}[A generator matrix of the $\lbrack 8,4 \rbrack$ extended Hamming code over $\mathbb{Z}_2$] \label{ex-exham}
Let
  \begin{equation*}
    G = \widetilde{N}_{4}= \begin{pmatrix}
      1 & 0 & 0 & 0 & 0 & 1 & 1 & 1 \\
      0 & 1 & 0 & 0 & 1 & 0 & 1 & 1 \\
      0 & 0 & 1 & 0 & 1 & 1 & 0 & 1 \\
      0 & 0 & 0 & 1 & 1 & 1 & 1 & 0
    \end{pmatrix}.
  \end{equation*}
  Then $\rho_{0} = 6$, and the weight enumerators are listed in Table~\ref{table-exham}.
  \begin{table}[t!]
    \centering
    \begin{tabular}{c|l}
    $\gcd(q, \rho_{0})$ & $W_{G}(x, y; q)$ \\ \hline
    $1$ &
    \begin{minipage}{0.5\hsize}
    \vspace{-0.7em}
      {\begin{align*}
        &x^{8}+ 10(q - 1)x^{4}y^{4} + 16(q - 1)x^{3}y^{5} + 4(q - 1)(7q - 20)x^{2}y^{6}\\
        &+ 8(q - 1)(q^{2} - 6q + 10)xy^{7} + (q - 1)(q^{3} - 7q^{2} + 21q - 25)y^{8}
      \end{align*}}\mbox{}
      \vspace{-1.4em}
    \end{minipage}\\ \hline
    $2$ & 
    \begin{minipage}{0.5\hsize}
    \vspace{-0.7em}
      {\begin{align*}
        &x^{8} + 2(5q - 3)x^{4}y^{4} + 16(q - 2)x^{3}y^{5} + 4(q - 2)(7q - 13)x^{2}y^{6} \\
        &+ 8(q - 2)^{2}(q - 3)xy^{7} + (q^{4} - 8q^{3} + 28q^{2} - 46q + 29)y^{8}
      \end{align*}}\mbox{}
      \vspace{-1.4em}
    \end{minipage} \\ \hline
    $3$ &
    \begin{minipage}{0.5\hsize}
    \vspace{-0.7em}
      {\begin{align*}
        &x^{8} +  2(5q - 4)x^{4}y^{4} + 8(2q - 3)x^{3}y^{5} + 4(7q^{2} - 27q + 23)x^{2}y^{6} \\
        &+ 8(q^{3} - 7q^{2} + 16q - 11)xy^{7} + (q^{4} - 8q^{3} + 28q^{2} -46q + 27)y^{8}
      \end{align*}}\mbox{}
      \vspace{-1.4em}
    \end{minipage} \\ \hline
    $6$ &
    \begin{minipage}{0.5\hsize}
    \vspace{-0.7em}
      {\begin{align*}
        &x^{8} + 2(5q - 2)x^{4}y^{4} + 8(2q - 5)x^{3}y^{5} + 4(7q^{2} - 27q + 29)x^{2}y^{6} \\
        &+ 8(q^{3} - 7q^{2} + 16q - 13)xy^{7} + (q^{4} - 8q^{3} + 28q^{2} -46q + 31)y^{8}
      \end{align*}}\mbox{}
      \vspace{-1.4em}
    \end{minipage} \\
    \end{tabular}
    \caption{The weight enumerators in Example \ref{ex-exham} for each $q\in \mathbb{Z}_{>0}$.} \label{table-exham}
  \end{table}
\end{example}
When $q=2$, it is well known that the codewords of the $[8,4]$ extended
Hamming code have weights only $0$, $4$, and $8$.
However, when the weight enumerator is regarded as a quasi-polynomial in $q$, the coefficients of $x^{8-i}y^i$ $(4<i<8)$ in the $2$-constituent are nonzero polynomials, and each is divisible by $q-2$.
Thus, these weights disappear only when $q=2$, while they persist throughout the corresponding constituent as polynomial terms.

\subsection{Dependence on the choice of an integral generator matrix}\label{sec-ex-dependence}

The next examples illustrate that the quasi-polynomial structure is an invariant of the underlying integral generator matrix rather than of the code over a fixed residue ring.
In particular, the following two integral matrices may define the same code over $\mathbb Z_4$ while giving rise to different families of codes and different periodic behaviors.
\begin{example}[A generator matrix of the Kerdock code $\mathcal{K}_2$~\cite{HKCSS94}]\label{ex-kerdock2}
Let
\begin{equation*}
  G=\begin{pmatrix}
    3 & 1 & 0 & 0 \\
    3 & 0 & 1 & 0 \\
    3 & 0 & 0 & 1
  \end{pmatrix}.
\end{equation*}
Then $\rho_{0} = 3$, and
\begin{align*}
    W_G(x,y; q) =
    \left\{
    \begin{aligned}
    &\begin{multlined}
    x^{4} + 6(q - 1)x^{2}y^{2} + 4(q - 1)(q - 2)xy^{3} \\
    +(q-1)(q^{2} -3q + 3)y^{4}
    \end{multlined}
    & \quad\text{if } \gcd(q,\rho_0)=1,\\[0.5em]
    &\begin{multlined}
    x^{4} + 6x^{3}y + 12(q - 2)x^{2}y^{2} + 2(3q^{2} - 12q + 13)xy^{3} \\ 
    + (q -3)(q^{2} - 3q + 3)y^{4}
    \end{multlined}& \quad \text{if } \gcd(q,\rho_0)=3.
    \end{aligned}
    \right.
\end{align*}
\end{example}

\begin{example}[The incidence matrix of the $4$-cycle graph]\label{ex-4cycle}
Let
\begin{equation*}
  G= \begin{pmatrix}
    -1 & 1 & 0 & 0 \\
    -1 & 0 & 1 & 0 \\
    -1 & 0 & 0 & 1
  \end{pmatrix}.
\end{equation*}
Then $\rho_{0} = 1$, and
\begin{equation*}
    W_G(x,y; q) =
    x^{4} + 6(q - 1)x^{2} y^{2} + 4(q - 1)(q - 2)xy^{3} + (q - 1)(q^{2} - 3q + 3)y^{4}
\end{equation*}
for any $q\in\mathbb{Z}_{>0}$.
\end{example}
Note that the generator matrices of Examples~\ref{ex-kerdock2} and~\ref{ex-4cycle} define the same code over $\mathbb{Z}_4$.
However, they define different families $\{C_G(q)\}_{q>1}$.
Consequently, their weight enumerators viewed as quasi-polynomials have different periods.

\subsection{Counterexamples to the converses of Theorem~\ref{thm-sufficient-necessary-mindist}\textup{(2)} and \textup{(3)}}

Theorem~\ref{thm-sufficient-necessary-mindist} provides a sufficient condition and a necessary condition for the constancy of the minimum weight on each residue class determined by $\gcd(q,\rho_0)$.
The following examples demonstrate that the converses of these conditions do not hold.
\begin{example}\label{ex-2004}
Let
\begin{equation*}
  G=\begin{pmatrix}
    2 & 0\\
    0 & 4
  \end{pmatrix}.
\end{equation*}
Then $\rho_{0} = 4$, and
\begin{align*}
    W_G(x,y; q) =
    \left\{
    \begin{aligned}
    &\begin{multlined}
    x^{2} + 2(q - 1)xy + (q - 1)^2y^{2}
    \end{multlined}
    & \quad\text{if } \gcd(q,\rho_0)=1,\\[0.5em]
    &\begin{multlined}
    x^{2} + (q - 2)xy + \frac{1}{4}(q - 2)^2y^{2}
    \end{multlined}
    & \quad\text{if } \gcd(q,\rho_0)=2,\\[0.5em]
    &\begin{multlined}
    x^{2} + \frac{1}{4}(3q - 8)xy + \frac{1}{8}(q - 2)(q - 4)y^{2}
    \end{multlined}& \quad \text{if } \gcd(q,\rho_0)=4.
    \end{aligned}
    \right.
\end{align*}
\end{example}
We consider the case where $m=4$.
The coefficient of $xy$ has no integer zero in the class $\gcd(q, \rho_{0}) = 4$.
This implies that $d_q=d_4=1$ for any $q\in\mathbb{Z}_{>0}$ with $\gcd(q,\rho_0)=4$.
On the other hand, since the last elementary divisor of $G$ is $e_r=4$, we have $\gcd(4,e_r)=4$.
Therefore, Example~\ref{ex-2004} is a counterexample to the converse of Theorem~\ref{thm-sufficient-necessary-mindist} \textup{(2)}.

\begin{example}[A generator matrix of the Kerdock code $\mathcal{K}_{4}$~\cite{HKCSS94}]\label{ex-Kerdoc4}
Let
\begin{equation*}
  G=\begin{pmatrix}
    1 & 1 & 1 & 1 \\
    0 & 2 & 0 & 2 \\
    0 & 0 & 2 & 2
  \end{pmatrix}.
\end{equation*}
Then $\rho_{0} = 2$, and
\begin{align*}
    W_G(x,y; q) =
    \left\{
    \begin{aligned}
    &\begin{multlined}
    x^{4} + 6(q - 1)x^{2} y^{2} + 4(q - 1)(q - 2)xy^{3} \\
    +(q - 1)(q^{2} - 3q + 3)y^{4}
    \end{multlined}
    & \quad\text{if } \gcd(q,\rho_0)=1,\\[0.5em]
    &\begin{multlined}
    x^{4} + 3(q - 2)x^{2}y^{2} + (q - 2)(q - 4)xy^{3} \\
    +\frac{1}{4}(q^{3} - 4q^{2} + 12q - 12)y^{4}
    \end{multlined}& \quad \text{if } \gcd(q,\rho_0)=2
    \end{aligned}
    \right.
\end{align*}
\end{example}
We consider the case where $m=2$.
Although $m$ does not divide the first elementary divisor $e_1=1$ of $G$, we have $4=d_m>d_4=2$.
Therefore, Example \ref{ex-Kerdoc4} is a counterexample to the converse of Theorem~\ref{thm-sufficient-necessary-mindist} \textup{(3)}.

\subsection{Non-monotonicity of the degrees of the constituents}

Theorem~\ref{thm-degreesfm-onebyone} and Lemma~\ref{lem-degreesfm-onebyone} provide the smallest index $i$ at which each degree of $f_i^m(t)$ appears first.
However, they do not imply that the degrees of the constituents increase monotonically with the weight.
The following example shows that such monotonicity does not hold in general.
\begin{example}[$\widetilde{Z}_{5}$]\label{ex-Z5}
Let $G=\widetilde{Z}_{5}$. Then $\rho_{0} = 12$, and the weight enumerators are listed in Table \ref{table-Z5}.

\begin{table}[t!]
  \centering
  \scalebox{0.95}{
  \begin{tabular}{c|l}
    $\gcd(q, \rho_{0})$ & $W_{G}(x, y; q)$ \\ \hline
    $1$ & 
    \begin{minipage}{0.5\hsize}
    \vspace{-0.7em}
      {\begin{align*}
        &x^{11} + 10(q - 1)x^{7}y^{4} + 5(q - 1)(2q - 3)x^{5}y^{6} + 55(q - 1)x^{4}y^{7} \\
        &+ 5(q - 1)(q^{2} + 13q - 51)x^{3}y^{8} + 10(q - 1)(4q^{2} - 24q + 43)x^{2}y^{9} \\
        &+ (q - 1)(11q^{3} - 84q^{2} + 246q - 299)xy^{10}
        + (q - 1)(q^{4} - 10q^{3} + 40q^{2} - 80q + 75)y^{11}
      \end{align*}}\mbox{}
      \vspace{-1.4em}
    \end{minipage} \\ \hline
    $2$ &
    \begin{minipage}{0.5\hsize}
    \vspace{-0.7em}
      {\begin{align*}
        &x^{11} + 10(q - 1)x^{7}y^{4} + (10q^{2} - 25q + 26)x^{5}y^{6} + 55(q - 2)x^{4}y^{7} \\
        &+ 5(q^{3} + 12q^{2} - 64q + 73)x^{3}y^{8} + 10(q - 2)(4q^{2} - 20q + 27)x^{2}y^{9} \\
        &+ (q - 2)(11q^{3} - 73q^{2} + 184q - 177)xy^{10} + (q - 2)(q^{4} - 9q^{3} + 32q^{2} - 56q + 43)y^{11}
      \end{align*}}\mbox{}
      \vspace{-1.4em}
    \end{minipage} \\ \hline
    $3$ &
    \begin{minipage}{0.5\hsize}
    \vspace{-0.7em}
      {\begin{align*}
        &x^{11} + 10(q - 1)x^{7}y^{4} + 5(2q^{2} - 5q + 5)x^{5}y^{6} + 5(11q - 21)x^{4}y^{7} \\
        &+ 5(q^{3} + 12q^{2} - 64q + 71)x^{3}y^{8} + 10(4q^{3} - 28q^{2} + 67q - 53)x^{2}y^{9} \\
        &+ (11q^{4} - 95q^{3} + 330q^{2} - 545q + 349)xy^{10} + (q^{5} - 11q^{4} + 50q^{3} - 120q^{2} + 155q - 85)y^{11}
      \end{align*}}\mbox{}
      \vspace{-1.4em}
    \end{minipage}\\ \hline
    $4$ & 
    \begin{minipage}{0.5\hsize}
    \vspace{-0.7em}
      {\begin{align*}
        &x^{11} + 10(q - 1)x^{7}y^{4} + (10q^{2} - 25q + 28)x^{5}y^{6} + 5(11q - 24)x^{4}y^{7}\\
        &+ 5(q^{3} + 12q^{2} - 64q + 77)x^{3}y^{8} + 10(4q^{3} - 28q^{2} + 67q - 56)x^{2}y^{9} \\
        &+ (11q^{4} - 95q^{3} + 330q^{2} - 545q + 364)xy^{10} + (q^{5} - 11q^{4} + 50q^{3} - 120q^{2} + 155q - 88)y^{11}
      \end{align*}}\mbox{}
      \vspace{-1.4em}
    \end{minipage}\\ \hline
    $6$ & \begin{minipage}{0.5\hsize}
    \vspace{-0.7em}
      {\begin{align*}
        &x^{11} + 10(q - 1)x^{7}y^{4} + (10q^{2} - 25q + 36)x^{5}y^{6} + 5(11q - 32)x^{4}y^{7}\\
        &+ 5(q^{3} + 12q^{2} - 64q + 93)x^{3}y^{8} + 10(4q^{3} - 28q^{2} + 67q - 64)x^{2}y^{9} \\
        &+ (11q^{4} - 95q^{3} + 330q^{2} - 545q + 404)xy^{10} + (q^{5} - 11q^{4} + 50q^{3} - 120q^{2} + 155q - 96)y^{11}    
      \end{align*}}\mbox{}
      \vspace{-1.4em}
    \end{minipage}\\ \hline
    $12$ & \begin{minipage}{0.5\hsize}
    \vspace{-0.7em}
      {\begin{align*}
        &x^{11} + 10(q - 1)x^{7}y^{4} + (10q^{2} - 25q + 38)x^5y^{6} + 5(11q-34)x^{4}y^{7} \\
        &+ 5(q^{3} + 12q^{2} - 64q + 97)x^3y^{8} + 10(4q^{3} - 28q^{2} + 67q - 66)x^2y^{9}\\
        &+ (11q^{4} - 95q^{3} + 330q^{2} - 545q + 414)xy^{10} + (q^{5} - 11q^{4} + 50q^{3} - 120q^{2} + 155q - 98)y^{11}
      \end{align*}}
    \end{minipage}
  \end{tabular}
  }
\caption{The weight enumerators in Example \ref{ex-Z5} for each $q\in \mathbb{Z}_{>0}$.} \label{table-Z5}
\end{table}
\end{example}
In all the preceding examples, as $i$ increases, the degree in $q$ of the coefficient of $x^{n-i}y^i$ also increases.
In contrast, in Example~\ref{ex-Z5}, when $\gcd(q,\rho_0)\neq1$, the degrees of the coefficients are no longer monotone: the coefficient of $x^7y^4$ has degree $1$, the coefficient of $x^6y^5$ is zero as a polynomial, the coefficient of $x^5y^6$ has degree $2$, the coefficient of $x^4y^7$ has degree $1$, and so on.
Thus, although Theorem~\ref{thm-degreesfm-onebyone} determines the first occurrence of each degree, the sequence of degrees of the constituents does not need to be monotone.

\section*{Acknowledgments}
This work was supported by the Sumitomo Foundation Grant for Basic Science Research Projects (No.~2501013),
JST K Program Grant Number JPMJKP24U2,
and the Japan Society for the Promotion of Science (JSPS) KAKENHI Grants JP25K17298 and JP26K06894.

The authors thank the reviewers for their helpful suggestions.

\bibliography{reference}
\bibliographystyle{amsalpha}

\end{document}